\documentclass[11pt,a4paper]{article}      
\usepackage[margin=1in]{geometry} 
\usepackage{amsmath}
\usepackage{amsfonts}
\usepackage{amssymb}
\usepackage{amsthm}
\usepackage{bm}
\usepackage[T1]{fontenc}
\usepackage[utf8]{inputenc}
\usepackage{mathrsfs}%
\usepackage{textcomp}
\usepackage{xcolor}
\usepackage{hyperref}
\usepackage{mathtools}
\usepackage{graphicx} 
\usepackage[
backend=biber,
style=numeric,
sorting=nyt
]{biblatex}
\newtheorem{Thm}{Theorem}[section]
\newtheorem{Def}[Thm]{Definition}
\newtheorem{Lem}[Thm]{Lemma}
\newtheorem{Prop}[Thm]{Proposition}

\newtheorem{Rem}[Thm]{Remark}
\newtheorem{Cor}[Thm]{Corollary}
\title{Fine asymptotics of the magnetization of the annealed dilute Curie-Weiss model}

\author{Fabian Apostel\thanks{Osnabrück University, Germany, fabian.apostel@uni-osnabrück.de} \and Hanna Döring\thanks{Osnabrück University, Germany, hanna.doering@uni-osnabrueck.de} \and Kristina Schubert\thanks{TU Dortmund, Germany, kristina.schubert@tu-dortmund.de}}

\date{\today}

\begin{document}

\maketitle

\begin{abstract}
    We consider the dilute Curie-Weiss model of size $N$, which is a generalization of the classical Curie-Weiss model in which the dependency structure between the spins is encoded not by the complete graph but by the (directed) Erd\H{o}s-Rényi graph on $N$ vertices in which every edge appears independently with probability $p(N)$. In the high-temperature regime with an external magnetic field ($0<\beta<1,h\in\mathbb{R}$) we prove for $p^{3}N^{2}\to\infty$ sharp cumulant bounds for the magnetization under the annealed Gibbs measure, which imply a central limit theorem with rate, moderate and large deviation principles, a concentration inequality, a normal approximation with Cramér correction and mod-Gaussian convergence.
\end{abstract}

\section{Introduction}
\subsection{Description of the model}
The dilute Curie-Weiss model on an Erd\H{o}s-Rényi graph, which was introduced in \cite{bovier1993thermodynamics}, is a mean-field model in which the particles (which we from now on call spins) only interact when the corresponding edge in the Erd\H{o}s-Rényi graph is present.  It is thus regarded as a disordered ferromagnet according to Fröhlich's lecture \cite{frohlich1985mathematical}. We define the model on a directed Erd\H{o}s-Rényi graph $\mathcal{G}(N,p)$ on $N\in\mathbb{N}$ vertices, where every directed edge appears independently with probability $p=p(N)\in(0,1]$. For simplicity we allow for self-loops. We say that the spins $\sigma_i,\sigma_j\in \left\lbrace -1,1\right\rbrace$ at sites $i,j\in[N]:=\left\lbrace 1,\dots,N\right\rbrace$ in the Curie-Weiss model interact if the edge $(i,j)$ appears in $\mathcal{G}(N,p)$.  We assume the dense regime $pN\to\infty$ as $N\to\infty$ to ensure that almost surely a giant connected component is present in $\mathcal{G}(N,p)$. As it turns out, we need to strengthen this condition to $p^{3}N^{2}\to\infty$ as $N\to\infty$ in order to make our techniques work as the model seems to have a change of behavior when $p^{3}N^{2}$ is of constant order which was already pointed out in \cite{kabluchko2019fluctuations}. We comment on this in Remark \ref{Remarkp3N2}.

To define the Hamiltonian of this system, we consider the i.i.d. random variables $\varepsilon_{ij}$ with probability distribution $\mathbb{P}(\varepsilon_{ij}=1)=1-\mathbb{P}(\varepsilon_{ij}=0)=p$ for $i,j\in [N]$, indicating whether an edge $(i,j)$ is present or not.
The Hamiltonian for $\sigma=(\sigma_i)_{i\in[N]}\in \Omega_N:=\left\lbrace -1,1\right\rbrace^{N}$ reads
    \begin{equation*}
H_{N,\varepsilon,\beta,h}(\sigma)
:= -\frac{\beta}{2pN} 
    \sum_{i,j=1}^N \varepsilon_{ij} \sigma_i \sigma_j
    - h \sum_{i=1}^N \sigma_i .
\end{equation*}
Here $\beta\in\mathbb{R}_{> 0}$ is the inverse temperature and $h\in\mathbb{R}$ is the strength of the external magnetic field. 

We consider the annealed version of the dilute Curie-Weiss model, where we obtain the Gibbs measure by averaging out the disorder of the system in the Boltzmann weight $e^{-H_{N,\varepsilon,\beta,h}(\sigma)}$ and then normalizing with the corresponding partition function. The Gibbs measure is hence defined by
\begin{equation}
    \mu_{N,p,\beta,h}(\sigma):=\frac{\mathbb{E}_\varepsilon\left[e^{-H_{N,\varepsilon,\beta,h}(\sigma)}\right]}{\mathbb{E}_\varepsilon\left[Z_{N,\varepsilon,\beta}(h)\right]}=:\frac{\mathbb{E}_\varepsilon\left[e^{-H_{N,\varepsilon,\beta,h}(\sigma)}\right]}{Z_{N,p,\beta}(h)} \label{annealedprob}
\end{equation}
where 
\begin{equation*}
    Z_{N,\varepsilon,\beta}(h):=\sum_\sigma e^{-H_{N,\varepsilon,\beta,h}(\sigma)}
\end{equation*}
is the so-called (quenched) partition function of the dilute Curie-Weiss model, which ensures that the corresponding Gibbs measure is a probability measure. We denote by $\mathbb{E}_\varepsilon$ the expectation with respect to the randomness of the $\varepsilon_{ij}$s. We refer to $Z_{N,p,\beta}(h)=\mathbb{E}_\varepsilon[Z_{N,\varepsilon,\beta}(h)]$ as the annealed partition function. The key object of this article is the so-called finite-volume pressure of the annealed Curie-Weiss model:
\begin{equation}
    \psi_{N,p,\beta}(h):=\frac{1}{N}\log Z_{N,p,\beta}(h). 
    \label{finitevolumepress}
\end{equation}
As is customary in statistical mechanics, we consider the configuration space $\Omega_N$ to be a probability space with respect to the Gibbs measure $\mu_{N,p,\beta,h}$. We prove bounds for the cumulants of the magnetization 
\begin{equation*}
    M_N(\sigma):=\sum_{i=1}^N\sigma_i,~~~~ \text{ for }\sigma\in\Omega_N,
\end{equation*}
on that probability space in the high-temperature regime with an external magnetic field ($0<\beta<1,h\in\mathbb{R}$). These bounds are sufficiently sharp in order to deduce quantitative results for the normal approximation after standardization, such as a quantitative central limit theorem and mod-Gaussian convergence. See Section \ref{mainresults}, in particular Corollaries \ref{Implications} and \ref{cor:largedeviation} for the complete list of the precise statements.

Observe that the case $p=1$ recovers the classical Curie-Weiss model. We refer to the textbooks \cite{bovier2006statistical,ellis2012entropy,friedli2017statistical} for a more detailed account of that model. 

For the dilute Curie-Weiss model, i.e. allowing $p$ to be $N$-dependent, in the annealed setting the model can be viewed as a classical Curie-Weiss model, where the inverse temperature is $N$-dependent; see Section \ref{Technical Preparation}. To control these parameters we impose the condition $p^{3}N^{2}\to\infty$ as $N\to\infty$, following \cite{kabluchko2019fluctuations}. It remains open whether our results extend to the entire dense regime $pN\to\infty$ as $N\to\infty$. 

\subsection{Annealed vs. randomly quenched vs. averaged quenched} \label{sec:annealedvsrqvsaq}
    Besides the annealed measure $\mu_{N,p,\beta,h}$, there are two other settings that are usually considered to study disordered systems: the randomly quenched and the averaged quenched settings. Since, in the annealed setting, we first average out the disorder on the level of the Boltzmann weight $e^{-H_{N,\varepsilon,\beta,h}}$ and then normalize, the spin system is not conditioned on a particular realization of the random graph, but incorporates all realizations of the disorder simultaneously. In the quenched setting, on the other hand, fixed configurations of the Erd\H{o}s-Rényi graph are considered and it thus concerns thermodynamic properties almost surely or with high probability. The Gibbs measure reads in this case:
    \begin{equation*}
        \mu_{N,\varepsilon,\beta,h}(\sigma):=\frac{e^{-H_{N,\varepsilon,\beta,h}(\sigma)}}{Z_{N,\varepsilon,\beta}(h)}.
    \end{equation*}
    The averaged quenched setting essentially flips the order of normalizing and averaging of the annealed setting and thus considers the average over spin systems, each being exposed to a fixed realization of the underlying random graph. In this case, the Gibbs measure is defined by
    \begin{equation*}
        \mu^{\text{av}}_{N,\varepsilon,\beta,h}(\sigma):=\mathbb{E}_{\varepsilon}\left[\frac{e^{-H_{N,\varepsilon,\beta,h}(\sigma)}}{Z_{N,\varepsilon,\beta}(h)}\right].
    \end{equation*}
    See \cite{can2019annealed,can2022annealed,giardina2016annealed} for a more detailed account of these settings. 
    
    Usually the quenched settings are most relevant for physicists while the annealed setting often is easier to handle. The latter, however, often already serves as a good approximation to the randomly quenched setting, since the almost sure limit of the randomly quenched pressure usually matches the limit of the annealed pressure.

    The present work thus focuses on the annealed setting which, when the underlying Erd\H{o}s-Rényi graph is dense, reveals a useful mean-field structure similar to the classical Curie-Weiss model as shown in Section \ref{Technical Preparation}. In particular, we show in Remark \ref{Remarkp3N2} that the partition function of the dilute Curie-Weiss model in the annealed setting has, up to a prefactor, the same form as in the classical Curie-Weiss model with an $N$-dependent inverse temperature. That alternative representation of the partition function is crucial for our complex-analytic approach, since it permits an analytic continuation in the strength of the external field which ultimately allows us to obtain precise bounds on derivatives of the pressure and hence cumulants of arbitrary order. The reduction of the annealed partition function to a Curie-Weiss-type partition function with an $N$-dependent effective inverse temperature is specific to the annealed setting. In contrast, the subsequent mechanism of deriving cumulant bounds from locally uniform control of the pressure is not intrinsically tied to annealing and could, in principle, also be applied in quenched settings once a corresponding uniform control of the corresponding pressure is available.

\subsection{Main results}\label{mainresults}
We want to go beyond a qualitative annealed central limit theorem in the setting of a dense Erd\H{o}s-Rényi graph and deduce additional quantitative bounds for the normal approximation. The central tools for this are the so-called cumulants (or semi-invariants). From now on, we always assume that $j\in\mathbb{N}$.
\begin{Def}
    The $j$-th cumulant of a real random variable $X$ is defined as
    \begin{equation*}
        \kappa_j(X):=\frac{d^j} {d t^j} \log \mathbb{E}\bigl[e^{tX}] \Big|_{t=0}.
    \end{equation*}
\end{Def}
Bounds on cumulants directly translate to quantitative bounds for the normal approximation. The most famous bound is the so-called Statulevi\v{c}ius condition which was first mentioned in \cite{rudzkis1978general}:
\begin{Def}
    We say that a real random variable $X$ with $\mathbb{E}[X]=0$ and $\mathbb{V}[X]=1$ satisfies the Statulevi\v{c}ius condition if 
    \begin{equation*}
        \vert\kappa_j(X)\vert\leq \frac{(j!)^{1+\gamma}}{\Delta^{j-2}}~~~~~~~~~(\forall j\geq3, j\in\mathbb{N})
    \end{equation*}
    for $\gamma\geq0$ and $\Delta>0$. \label{statulevicius}
\end{Def}
\begin{Rem}\label{remarkmomentscumulants}
    There is a well-known relation between (central) moments and cumulants which can be made explicit using the cumulant generating function. We refer to \cite{leonov1959method} for more details on that. In particular, we find for a real random variable $X$
    \begin{equation*}
        \kappa_1(X)=\mathbb{E}[X],~~~~~~\kappa_2(X)=\mathbb{V}(X).
    \end{equation*}
    Moreover, cumulants of order $j\geq2$ are translation invariant and homogeneous of order $j$, i.e. for $c,d\in\mathbb{R}$ we find for $j\geq2$
    \begin{equation*}
        \kappa_j(cX+d)=c^{j}\kappa_j(X).
    \end{equation*}
\end{Rem}

From now on, throughout the rest of this article, we will denote by $\mathbb{E}$ and $\mathbb{V}$ the expectation and the variance with respect to the annealed Gibbs measure $\mu_{N,p,\beta,h}$. The main goal of this article is to prove the following theorem.
\begin{Thm}[Statulevi\v{c}ius condition for the magnetization for the annealed Curie-Weiss model]
    Suppose that $h,\beta\in\mathbb{R}$ with $0<\beta<1$ and $p^{3}N^{2}\to\infty$ as $N\to\infty$.
    Then, for all sufficiently large $N$, the standardized magnetization $m_N:=\frac{M_N-\mathbb{E}\left[M_N\right]}{\sqrt{\mathbb{V}\left[M_N\right]}}$ satisfies the Statulevi\v{c}ius condition under the annealed measure $\mu_{N,p,\beta,h}$. In particular, for any $R$ with $0<R<\arccos\left(\sqrt{\beta}\right)-\sqrt{\beta(1-\beta)}$ there exist  $N_0\in\mathbb{N}$ and $C,c>0$ such that for $N\geq N_0$ and any $j\geq3$
    \begin{eqnarray*}
        \left\vert\kappa_j\left(m_N\right)\right\vert&\leq& C \frac{j!}{\left(Rc\sqrt{N}\right)^{j-2}}\\
        &\leq& \frac{j!}{\left((R/\tilde{C})\sqrt{N}\right)^{j-2}},
    \end{eqnarray*}
    where $\tilde{C}:=\max(C/c,1)$. We have $\mathbb{E}\left[M_N\right]=N(m(h)+o(1))$ and $\mathbb{V}\left[M_N\right]= N(\chi(h)+o(1))$, where $m(h)$ is the unique solution of the \textit{mean-field equation}
    \begin{equation*}
        m(h)=\tanh(h+\beta m(h)),
    \end{equation*}
    and 
    \begin{equation*}
        \chi(h):=\frac{1-m^{2}(h)}{1-\beta(1-m^{2}(h))}
    \end{equation*}
    is the susceptibility.
    \label{Statuleviciusthmcw}
\end{Thm}
\begin{Rem}
    In Theorem \ref{Statuleviciusthmcw} the choice of $R$ depends on $\beta$ but not on $h$, whereas $C$, $m(h)$ and $\chi(h)$ depend on both $\beta$ and $h$.
\end{Rem}
To derive the bounds for the cumulants, we require an asymptotic analysis of the partition function $Z_{N,p,\beta}(h)$ as $N\to\infty$. The standard approach for mean-field models is the Laplace method (see e.g. \cite{kirsch2019curie}). This method is only applicable if $h$ is real. However, since the $j$-th order cumulants are given by the $j$-th derivative of the pressure which we can then bound by Cauchy's estimate (see \eqref{cumulantbound} and \cite[Chapter 2; Corollary 4.3.]{stein2010complex}), we require $h$ to be complex, so we instead employ the saddle-point method (see e.g. \cite[Chapter 45.4, Theorem 2]{SidorovFedoryukShabunin1985}). Moreover, the asymptotics of $Z_{N,p,\beta}(h)$ must hold locally uniformly for $h$ on some suitable subset of the complex plane. Verifying the assumptions of the method and carrying out the analysis constitutes the main technical difficulty of our approach.

Since by Theorem \ref{Statuleviciusthmcw} all cumulants of order $j\geq 3$ tend to zero for $N\to \infty$, we find by \cite{SvanteJansonSemi} (or equivalently by the moment method) that a (qualitative) central limit theorem holds. Before stating these results, we add the definitions of the moderate and large deviation principles:
\begin{Def}{\cite[Definition III.5 and III.6]{hollander2000large}}\label{def:largedeviation}
    Let $\mathcal{X}$ be a Polish space. A sequence of probability measures $(P_n)_{n\in\mathbb{N}}$ on $\mathcal{X}$ is said to satisfy a large deviation principle with speed $a_n$ ($a_n\xrightarrow[]{n\to\infty}\infty$) and with rate function $I\colon \mathcal{X}\to[0,\infty]$ if
    \begin{enumerate}
        \item $I$ is lower semi-continuous, $I\not\equiv \infty$ and the level sets $\left\lbrace x\in\mathcal{X}\colon I(x)\leq c\right\rbrace$ are compact for all $0\leq c<\infty$.
        \item  $\limsup_{n\to\infty}\frac{1}{a_n}\log P_n(C)\leq-\inf_{x\in C}I(x)~~~\forall C\subset\mathcal{X}$ closed.
        \item $\liminf_{n\to\infty}\frac{1}{a_n}\log P_n(O)\geq-\inf_{x\in O}I(x)~~~\forall O\subset\mathcal{X}$ open.
    \end{enumerate}
    Moreover, we say that a sequence of random variables $(X_n)_{n\in\mathbb{N}}$ satisfies a large deviation principle if their distributions do. We speak of a moderate deviation principle whenever the scaling of the corresponding random variables lies between the scales of the central limit theorem and the law of large numbers.
\end{Def}
Following the book \cite{saulis2012limit}, the paper \cite{doring2013moderate} and the survey \cite{doring2022method}, the Statulevi\v{c}ius condition implies several quantitative results summarized in the following Corollary.
\begin{Cor}\label{Implications}
    Under the assumptions of Theorem \ref{Statuleviciusthmcw} there exists $N_0\in\mathbb{N}$ such that for all $N\geq N_0$ the following statements hold for $Z\sim\mathcal{N}(0,1)$:
    \begin{enumerate}
    \item \textit{Normal approximation with Cramér corrections.} For $x\in (0,c\sqrt{N})$ with suitable constant $c>0$ we have
    \begin{equation}
        \mu_{N,p,\beta,h}\left(m_N\geq x\right)=e^{\tilde{L}(x)}\mathbb{P}(Z\geq x)\left(1+O\left(\frac{x+1}{\sqrt{N}}\right)\right) \label{eq:cramercorrection}
    \end{equation}
    where $\tilde{L}$ is related to the so-called Cramér-Petrov series and satisfies $\vert\tilde{L}(x)\vert=$ \\$O(x^{3}/\sqrt{N})$.
    \item \textit{Bound on the Kolmogorov distance.} The following bound of Berry-Esseen type holds true:
    \begin{equation*}
        \sup_{x\in\mathbb{R}}\left\vert\mu_{N,p,\beta,h}(m_N\leq x)-\mathbb{P}(Z\leq x)\right\vert\leq\frac{c}{\sqrt{N}}
    \end{equation*}
    for some constant $c>0$. 
    \item \textit{Concentration inequality.}  One has
    \begin{equation*}
        \mu_{N,p,\beta,h}(m_N\geq x)\leq \exp\left(-\frac{1}{2}\frac{x^{2}}{2C+x^{2}/(Rc\sqrt{N})}\right),
    \end{equation*}
    where $C,c$ and $R$ are the constants from Theorem \ref{Statuleviciusthmcw}.
    \item[4.] \textit{Moderate deviations.}  For every sequence $(a_N)_{N\in\mathbb{N}}$ with $a_N\to\infty$ as $N\to\infty$ and 
    \begin{equation*}
        a_N=o\left(\sqrt{N}\right),
    \end{equation*}
    the sequence $(m_N/a_N)_{N\in\mathbb{N}}$ satisfies the moderate deviation principle with speed $a_N^{2}$ and rate function $I(x)=x^{2}/2$.
    \item[5.] \textit{Mod-Gaussian convergence.} 
    Set $c_3:=\lim_{N\to\infty}\kappa_3(m_N)(R/\tilde{C})\sqrt{N}$. We find that the random variable $Y_N:=\left((R/\tilde{C})\sqrt{N}\right)^{1/3}m_N$ converges mod-Gaussian, i.e. setting $\eta(z):=z^{2}/2$ and $\phi(z):=\exp(c_3z^{3}/6)$, we find
    \begin{equation*}
        \lim_{N\to\infty}\exp\left(-\left((R/\tilde{C})\sqrt{N}\right)^{2/3}\eta(is)\right)\mathbb{E}\left[e^{isY_N}\right]=\phi(is)
    \end{equation*}
    holds, locally uniformly, for $s\in\mathbb{R}$.
\end{enumerate}
\end{Cor}

The conclusions in Corollary \ref{Implications} provide information on different fluctuation scales of the magnetization $M_N$. On the natural scale $\sqrt{N}$ of central limit theorems we get the optimal rate $N^{-1/2}$ for the Kolmogorov distance describing the speed of convergence to the standard normal, as well as a Bernstein-type concentration inequality providing tail bounds. 

A moderate deviation principle controls probabilities of deviations that are larger than the natural scale of the central limit theorem but smaller than the macroscopic scale of large deviations, i.e. larger than $\sqrt{N}$ and smaller than $N$, interpolating central limit fluctuations and large deviation behavior. Moreover, Proposition \ref{asymptoticprop} implies together with the Gärtner-Ellis theorem \cite[Theorem V.6]{hollander2000large} a large deviation principle.
The proof of the following corollary will be provided after the proof of Theorem \ref{Statuleviciusthmcw} in Section \ref{Proofofthm}.
\begin{Cor}\label{cor:largedeviation}
    Under the assumptions of Theorem \ref{Statuleviciusthmcw} the sequence $(M_N/N)_{N\in\mathbb{N}}$ satisfies a large deviation principle with speed $N$ whose rate function $I(x)$ is given by the Legendre transform of $\Phi(s(h+x),h+x)-\Phi(s(h),h)$, where $\Phi(\cdot,\cdot)$ and $s(\cdot)$ are defined in Proposition \ref{asymptoticprop} below.
\end{Cor}
This article thus provides deviation principles on all scales from the microscopic scale $N^{-1/2}$ to the macroscopic scale $N^{-1}$.

While the central limit theorem only ensures Gaussian behavior for fixed inputs of the distribution function, i.e. for $\mu_{N,p,\beta,h}(m_N> x)$, where $x$ is fixed, the Cramér correction further refines the tail probabilities by quantifying, via the factor $e^{\tilde{L}(x)}$, the correction to the standard normal tail which in this case only becomes relevant when $x$ is $N$-dependent on the scale $o(\sqrt{N})$. The coefficients of the series $\tilde{L}(x)$ are expressed via the cumulants of $m_N$ in a rather complicated way. However, the bound on $\vert \tilde{L}(x)\vert$ stated in Corollary \ref{Implications} leads us to the conclusion that the choice $x=o(N^{1/6})$ lets the tail probability of the standard normal dominate the approximation in \eqref{eq:cramercorrection}. That regime is commonly referred to as the zone of normal convergence. The choice $x=N^{1/6}$ consequently is the first threshold that makes the factor $e^{\tilde{L}(x)}$ potentially non-negligible in \eqref{eq:cramercorrection} with $\tilde{L}(x)$ possibly being of constant order or diverging in $N$. The same holds true when choosing $x\in(0,c\sqrt{N})$ beyond the zone of normal convergence. Note that for $x=o(\sqrt{N})$ the relative error still converges to 0, while becoming non-trivial for $x=c\sqrt{N}$. \cite{petrov1995limit}

Mod-Gaussian convergence is a complementary result to the Cramér correction which was introduced in \cite{JacodKowalskiNikeghbali+2011+835+873} and later generalized to the notion of mod-$\varphi$ convergence in \cite{delbaen2015mod} to study, besides quantitative results in probability theory, also problems in analytic number theory and random matrix theory. Instead of answering the question how to correct the tail probability of the standard normal to obtain estimates for the tail probabilities of the random variable in question on different scales, mod-Gaussian convergence answers the question how to scale this random variable such that, after removing the leading standard normal contribution from its characteristic function, a non-trivial limiting function remains. In our case, the scaling is of order $N^{1/6}$ which matches the boundary of the zone of normal convergence above.

As mentioned above, mod-Gaussian convergence yields a framework to study quantitative aspects. This entails, for example, moderate deviations as well as a local limit theorem. The latter requires some additional assumptions on the mod-Gaussian convergence stated in Corollary \ref{Implications}; see \cite{delbaen2015mod}. In this article we do not pursue this any further. However, we expect such a result to be within reach for the present model. The interested reader might consult the book \cite{feray2016mod} and the articles \cite{JacodKowalskiNikeghbali+2011+835+873,meliot2015mod} for a more detailed account of mod-Gaussian convergence.

The authors also expect that results analogous to Theorem \ref{Statuleviciusthmcw} and Corollaries \ref{Implications} and \ref{cor:largedeviation} carry over to the randomly quenched setting. Proposition \ref{asymptoticprop} below essentially states the local uniform convergence of the finite-volume annealed pressure to the infinite-volume pressure of the classical Curie-Weiss model. It turns out that this result is the key tool for proving Theorem \ref{Statuleviciusthmcw}. In \cite{bovier1993thermodynamics} it was shown that the randomly quenched pressure also converges pointwise in probability to the infinite-volume pressure of the classical Curie-Weiss model. We argue that in \cite{kabluchko2019fluctuations} the randomly quenched pressure was successfully approximated by the annealed pressure at 0. Lifting this approximation to hold locally uniformly in probability would then imply that Theorem \ref{Statuleviciusthmcw} holds in probability with $\mu_{N,p,\beta,h}$ replaced by $\mu_{N,\varepsilon,\beta,h}$. To make this more explicit, recall that $\psi_{N,1,\beta}(h)$ recovers the finite-volume pressure of the classical Curie-Weiss model. We define $\psi_{1,\beta}(h):=\lim_{N\to\infty}\psi_{N,1,\beta}(h)$ to be the infinite-volume pressure of the classical Curie-Weiss-model, whose existence is a classical result (see e.g. \cite{friedli2017statistical}). Moreover, we define the finite-volume pressure of the randomly quenched dilute Curie-Weiss model as
\begin{equation*}
    \psi_{N,\varepsilon,\beta}(h):=\frac{1}{N}\log Z_{N,\varepsilon,\beta}(h).
\end{equation*}
In light of Proposition \ref{asymptoticprop} and \cite{bovier1993thermodynamics} we would like to prove that
\begin{equation*}
    \psi_{N,\varepsilon,\beta}(h)\xrightarrow[]{N\to\infty}\psi_{1,\beta}(h)
\end{equation*}
in probability, locally uniformly on some complex strip around the real axis. Note that, by the triangle inequality
\begin{equation*}
    \vert \psi_{N,\varepsilon,\beta}(h)-\psi_{1,\beta}(h)\vert\leq \vert \psi_{N,\varepsilon,\beta}(h)-\psi_{N,p,\beta}(h)\vert+\vert \psi_{N,p,\beta}(h)-\psi_{1,\beta}(h)\vert.
\end{equation*}
Showing that the second summand on the right-hand side converges locally uniformly to 0 is the content of Proposition \ref{asymptoticprop}. To show that the first summand on the right-hand side converges in probability locally uniformly to 0, it suffices to prove that the ratio of the involved partition functions is close to 1 in probability locally uniformly since
\begin{equation}
    \vert\psi_{N,\varepsilon,\beta}(h)-\psi_{N,p,\beta}(h)\vert=\left\vert\frac{1}{N}\log\frac{Z_{N,\varepsilon,\beta}(h)}{Z_{N,p,\beta}(h)}\right\vert.\label{eq:ratioquenched}
\end{equation}
The ratio inside of the logarithm on the right-hand side of \eqref{eq:ratioquenched} has been recognized to converge in probability to 1 at $h=0$ in \cite[Theorem 2]{kabluchko2019fluctuations}. Moreover, observe that for $h\in\mathbb{C}$ we have
\begin{eqnarray}
    \mathbb{E}_\varepsilon\left\vert\frac{Z_{N,\varepsilon,\beta}(h)}{Z_{N,p,\beta}(h)}-1\right\vert^{2}&=&\mathbb{E}_\varepsilon\left\vert \frac{Z_{N,\varepsilon,\beta}(h)}{Z_{N,p,\beta}(h)}\right\vert^{2}- 2 \operatorname{Re}\mathbb{E}_\varepsilon\frac{Z_{N,\varepsilon,\beta}(h)}{Z_{N,p,\beta}(h)}+1\nonumber\\
    &=& \frac{\mathbb{E}_\varepsilon\left\vert Z_{N,\varepsilon,\beta}(h)\right\vert^{2}}{\left\vert Z_{N,p,\beta}(h)\right\vert^{2}}-1,\label{eq:quenchedl2}
\end{eqnarray}
since $\mathbb{E}_\varepsilon Z_{N,\varepsilon,\beta}(h)=Z_{N,p,\beta}(h)$ by definition. The present article provides in Section \ref{Proofofprop} a precise locally uniform asymptotic description of $Z_{N,p,\beta}(h)$ on a suitable strip in the complex plane around the real axis, which appears squared in the denominator of equation \eqref{eq:quenchedl2}. The authors conjecture that the numerator in \eqref{eq:quenchedl2} is asymptotic to the square of this asymptotic description, up to a relative error that converges locally uniformly in probability to 0, which would prove the desired locally uniform convergence in probability of $\frac{Z_{N,\varepsilon,\beta}(h)}{Z_{N,p,\beta}(h)}$ to 1. However, we have not been able to provide a proof of that conjecture yet.

\subsection{Discussion}
It is well known that the magnetization in the classical Curie-Weiss model, i.e. $p=1$, exhibits a (qualitative) central limit theorem at high temperature, $\beta<1$, and for every $h\in\mathbb{R}$ \cite{ellis1978limit,ellis1978statistics,kirsch2019curie}.

All of the fluctuation results stated in Corollaries \ref{Implications} and \ref{cor:largedeviation} have previously been established for the classical Curie-Weiss model, i.e. for $p=1$, although by different methods.
A bound on the Kolmogorov distance was obtained in \cite{eichelsbacher2010stein}. 
A Cramér correction to the normal approximation was given in \cite{Lchen2013}. A concentration inequality was proved in \cite{chatterjee2006stein}.
The large deviation principle is a classical result; see \cite{friedli2017statistical}. 
Mod-Gaussian convergence was established in \cite{meliot2015mod} together with a local limit theorem. 
Moreover, mod-Gaussian convergence implies a moderate deviation principle (see \cite{feray2016mod}) which was proved previously in \cite{EichelsbacherLoewe2004}. 

For the dilute Curie-Weiss model with $p^{3}N^{2}(\log N)^{-1}\to\infty$ whenever $0<\beta<1$ and $h=0$ or $p^{2}N\to\infty$ whenever $(\beta,h)\in\left\lbrace (a,b)\in\mathbb{R}^{2}\colon a>0,h\neq0\right\rbrace\cup\left\lbrace (a,0)\in\mathbb{R}^{2}\colon a>1\right\rbrace$ a bound on the Kolmogorov distance as well as a concentration result has been proved in both the annealed and averaged quenched setting in \cite{DebMukherjee}. We were able to recover this result at least in the first parameter regime and the regime $0<\beta<1$ and $h\neq0$ for the less dense graph $p^{3}N^{2}\to\infty$.
Moreover, to the best of our knowledge, no other quantitative results implied by the Statulevi\v{c}ius condition have been established in the disordered settings presented in Section \ref{sec:annealedvsrqvsaq} of the dilute Curie-Weiss model. 
In addition, our results confirm the results in the classical case $p=1$, thereby providing a unified cumulant-based derivation of the refined fluctuation results mentioned above.

On sparse Erd\H{o}s-Rényi graphs, an annealed central limit theorem was proved in \cite{giardina2016annealed}. On dense Erd\H{o}s-Rényi graphs, an annealed central limit theorem was not explicitly proved, but can be deduced from the techniques developed in \cite{kabluchko2019fluctuations,kabluchko2020fluctuations,kabluchko2022fluctuations} at least in the absence of an external magnetic field. Further qualitative central limit theorems were proved for the randomly quenched measure in \cite{giardina2015quenched,kabluchko2019fluctuations} on sparse and dense Erd\H{o}s-Rényi graphs, respectively.

The dilute Curie-Weiss model considered here belongs to the broader class of Ising models on random graphs. A different, extensively studied class of these models consists of Ising models on locally tree-like graphs for which a general treatment was initiated in \cite{dembo2010}. Two notable subclasses are Ising models on configuration models and Ising models on generalized random graphs under suitable assumptions. The interested reader might consult \cite{van2017random} for a detailed introduction of these random graphs.

For the Ising model on generalized random graphs, qualitative annealed and randomly quenched central limit theorems for the magnetization were proved in \cite{giardina2016annealed,giardina2015quenched} in the uniqueness regimes $\beta\geq0$ and $h\neq0$, on the one hand, and $0<\beta<\beta_c^{\text{grg,an}}$ (resp. $\beta_c^{\text{grg,qu}}$) and $h\in\mathbb{R}$, on the other hand. Here, $\beta_c^{\text{grg,an}}$ and $\beta_c^{\text{grg,qu}}$ denote the critical inverse temperatures of the corresponding models. In addition, in the annealed setting, a bound on the Kolmogorov distance was established in \cite{DOMMERS2020605} in the same regimes. Moreover, in this setting, a large deviation principle was proved for the magnetization in \cite{dommers2018large} for the high-temperature regimes with an external magnetic field $0<\beta<\beta_c^{\text{grg,an}}$ (resp. $\beta_c^{\text{grg,qu}}$) and $h\in\mathbb{R}$. 

There are also qualitative annealed central limit theorems for the magnetization in the Ising model on configuration models in which the degree sequence is constant or bounded. In the setting of constant degrees, qualitative annealed central limit theorems were proved for 2-regular random graphs in \cite{giardina2016annealed} for the whole parameter regime $\beta\geq0$ and $h\in\mathbb{R}$ (due to the absence of a phase transition in this model) and were later generalized to all random $d$-regular graphs in \cite{can2019annealed} in the uniqueness regimes $\beta\geq0$ and $h\neq0$ on the one hand, and $0<\beta<\beta_c^{\text{rrg}}$ and $h\in\mathbb{R}$, on the other hand. In the latter regime, $\beta_c^{\text{rrg}}$ denotes the corresponding critical inverse temperature defined in \cite{can2019annealed}. Further qualitative annealed central limit theorems were proved in \cite{giardina2016annealed} for the configuration model in which the degrees of the underlying random graph are either 1 or 2, again for all $\beta\geq0$ and $h\in\mathbb{R}$ (for the same reason as for the 2-regular random graphs).

In \cite{giardina2015quenched} qualitative randomly quenched and averaged quenched central limit theorems were proved for the Ising model on 2-regular graphs and on the configuration model in which the degrees of the underlying random graph are either 1 or 2. Again, these central limit theorems hold for all parameters $\beta\geq0$ and $h\in\mathbb{R}$.

The proofs of the qualitative central limit and large deviation results for locally tree-like random graphs mentioned above rely heavily on proving the pointwise convergence of the cumulant generating function and its derivatives and thus on proving the existence of the pressure and its derivatives for a fixed strength of the external magnetic field in the thermodynamic limit. They are therefore closely related to the proof strategy adopted in the present article (see Lemma \ref{lem:representationcumulant} and Proposition \ref{asymptoticprop} below). 

So far, besides the large deviation result in \cite{dommers2018large}, quantitative results for the normal approximation, as stated in Corollaries \ref{Implications} and \ref{cor:largedeviation}, are not available for the entire class of locally tree-like random graphs, although thermodynamic properties have been established for large subclasses. In \cite{can2022annealed,dembo2010,dommers2010ising}, for example, the pointwise convergence of the pressure has been established in both the annealed and quenched settings. Following the reasoning below, we conjecture that lifting this convergence to locally uniform convergence in a suitable subset of the complex plane around the real axis would yield these fluctuation results together with fine asymptotics analogous to the results presented in Corollaries \ref{Implications} and \ref{cor:largedeviation}. Moreover, the tools used in \cite{can2019annealed,giardina2016annealed,giardina2015quenched} to prove the pointwise convergence of the pressure share similarities to those used below, whereas in \cite{giardina2015quenched,giardina2016annealed} the thermodynamic limit of the pressure has been recognized to be the thermodynamic limit of the pressure in the one-dimensional classical Ising model, which is holomorphic in the strength of the external magnetic field $h$, making it plausible that such a local uniform convergence holds in these cases.

In addition, we also conjecture that the approach presented in this article is applicable to spin models outside the realm of disordered spin systems, provided, they admit a mean-field structure. This would for example be the case for the Potts model or the block spin Ising model.

This article is organized as follows. In Section \ref{Technical Preparation}, we derive an alternative representation of the partition function enabling us to treat this model similarly to a mean-field model and provide further results on the parameters that appear. At the beginning of Section \ref{Sec3}, we provide a short outline of the proof of Theorem \ref{Statuleviciusthmcw}. In Section \ref{Proofofthm}, we establish in Proposition \ref{asymptoticprop} that the (annealed) finite-volume pressure converges locally uniformly in a suitable strip in the complex plane to the infinite-volume pressure of the classical Curie-Weiss model and then use this result to prove Theorem \ref{Statuleviciusthmcw}. We conclude in Section \ref{Proofofprop} by proving the aforementioned proposition.

\section{Technical Preparation}\label{Technical Preparation}
In this article, the main object for studying the cumulants of the magnetization is the finite-volume pressure $\psi_{N,p,\beta}(h)=N^{-1}\log Z_{N,p,\beta}(h)$ of the annealed dilute Curie-Weiss model defined in \eqref{finitevolumepress}. To prove Theorem \ref{Statuleviciusthmcw} we are going to use the following result:
\begin{Lem}{(\cite[Exercise 3.4]{friedli2017statistical})}\\
    The cumulant generating function satisfies \label{lem:representationcumulant}
    \begin{equation}
    \log\mathbb{E}\left[e^{tM_N}\right]=N\cdot\left(\psi_{N,p,\beta}(h+t)-\psi_{N,p,\beta}(h)\right).\label{representationcumulantgen}
    \end{equation}
    In particular, we have for the $j$-th cumulant at $h=h_0\in\mathbb{R}$
    \begin{equation}
        \kappa_j\left(M_N\right)=N\cdot\left.\frac{\operatorname{d}^j}{\operatorname{d}h^j}\psi_{N,p,\beta}(h)\right\vert_{h=h_0}.\label{representationcumulanderiv}
    \end{equation}
\end{Lem}
\begin{proof}
    Since we take the expectation inside the logarithm with respect to the Gibbs measure, we have that the left-hand side of \eqref{representationcumulantgen} equals 
    \begin{eqnarray*}
        \log\frac{Z_{N,p,\beta}(h+t)}{Z_{N,p,\beta}(h)}&=&N\!\left(\frac{1}{N}\log Z_{N,p,\beta}(h+t)-\frac{1}{N}\log Z_{N,p,\beta}(h)\right)\\
        &=&\!N\!\left(\psi_{N,p,\beta}(h+t)-\psi_{N,p,\beta}(h)\right).
    \end{eqnarray*}
    The relation to the $j$-th cumulant in \eqref{representationcumulanderiv} now follows by an inductive argument using the chain rule. 
\end{proof}
The techniques used below require an integral representation of the partition function $Z_{N,p,\beta}(h)$.
We find that
\begin{eqnarray}
Z_{N,p,\beta}(h)&=&\mathbb{E}_{\varepsilon}[ Z_{N,\varepsilon,\beta}(h)]
= \sum_{\sigma}
   \mathbb{E}_{\varepsilon}\!\left[
      e^{\frac{\beta}{2pN}
           \sum_{i,j=1}^N \varepsilon_{ij}\sigma_i\sigma_j }
   \right]
   e^{h \sum_i\sigma_i}\nonumber\\
   &=&\sum_{\sigma}\!\left[\prod_{i,j=1}^{N}
   \mathbb{E}_{\varepsilon}\left[
      e^{\frac{\beta}{2pN}
           \varepsilon_{ij}\sigma_i\sigma_j }\right]
   \right]
   e^{h \sum_i\sigma_i}.
   \label{annealedpartition}
\end{eqnarray}
Moreover,
\begin{equation*}
\mathbb{E}_{\varepsilon}\!\left[
 e^{\frac{\beta}{2pN} \varepsilon_{ij} \sigma_i \sigma_j}
\right]
= (1-p) + p\, e^{ \frac{\beta}{2pN} \sigma_i\sigma_j }.
\end{equation*}
In order to apply the Hubbard-Stratonovich transform to the summands in \eqref{annealedpartition}, we use a technique presented in \cite{kabluchko2019fluctuations,kabluchko2020fluctuations,kabluchko2021fluctuations,kabluchko2022fluctuations}: Since $(1-p) + p\, e^{\frac{\beta}{2pN}s}$ is positive for $s\in\left\lbrace -1,1\right\rbrace$, there exists a unique solution $(a,b)\in\mathbb{R}^{2}$ of
\begin{equation*}
a e^{\frac{b}{2pN}\, \sigma_i \sigma_j}
= (1-p) + p\, e^{\frac{\beta}{2pN}\sigma_i\sigma_j},
\end{equation*}
by \cite[Proof of Lemma 3.2]{kabluchko2020fluctuations}. Because $\sigma_i\sigma_j\in\left\lbrace -1,1\right\rbrace$, we have the following system of equations
\begin{equation*}
(1-p)+p e^{\frac{\beta}{2pN}}
   = a e^{\frac{b}{2pN}},
~~~~~~~
(1-p)+p e^{-\frac{\beta}{2pN}}
   = a e^{-\frac{b}{2pN}}.
\end{equation*}
Solving for $a$ and $b$ yields 
\begin{equation}
    a=\sqrt{\left((1-p)+pe^{-\frac{\beta}{2pN}}\right)\left((1-p)+pe^{\frac{\beta}{2pN}}\right)}\label{a}
\end{equation}
and 
\begin{equation}
\frac{b}{2pN} 
= \frac12 
   \log \frac{ (1-p) + p e^{\frac{\beta}{2pN}} }
              { (1-p) + p e^{-\frac{\beta}{2pN}} },\label{b}
\end{equation}
where we leave out the dependence on $N$ and $\beta$ in the notation of $a$ and $b$.
Plugging \eqref{a} and \eqref{b} into \eqref{annealedpartition} we find 
\begin{equation}
    Z_{N,p,\beta}(h)=a^{N^2}
   \sum_{\sigma}
   e^{\frac{b}{2pN}M_N(\sigma)^2
     + hM_N(\sigma)}. \label{annealdpartition2}
\end{equation}
\begin{Rem}
    Regarding the asymptotic treatment of $a^{N^{2}}$, observe that 
    \begin{equation*}
        \log a^{N^{2}}=\frac{\log a}{\frac{1}{N^{2}}}.
    \end{equation*}
    The change of variables $x=\frac{1}{N}$ changes the limit $N\to\infty$ to $x\to 0$. By \eqref{a} we find $\lim_{x\to0}\log a=0$. A short calculation using l'H\^opital's rule repeatedly then yields that for $pN\to \infty$ as $N\to\infty$
    \begin{equation*}
        a^{N^{2}}\sim \exp\left(\frac{(1-p)\beta^{2}}{8p}\right),
    \end{equation*}
    where we write $c_N\sim d_N$ for two sequences if their ratio converges to 1. \label{convergenceaN2}
\end{Rem}

\begin{Rem}
Setting 
\begin{equation*}
    \beta_{\text{eff}}=\beta_{\text{eff}}(N):=\frac{b}{p}, 
\end{equation*}
equation \eqref{annealdpartition2} shows, up to the prefactor $a^{N^{2}}$, that the partition function of the dilute Curie-Weiss model in the annealed setting has the same form as in the classical Curie-Weiss model with an $N$-dependent inverse temperature $\beta_{\text{eff}}$. The asymptotics of $a^{N^{2}}$ have already been treated in Remark \ref{convergenceaN2}. In Lemma \ref{convergencebp} below, we establish the convergence of $\beta_{\text{eff}}$ to $\beta$.

 It is exactly the asymptotic treatment of $\beta_{\text{eff}}$  (see \eqref{ratebetaeff}) that requires our additional condition $p^{3}N^{2}\to\infty$ as $N\to\infty$.\label{Remarkp3N2}
\end{Rem}
\begin{Lem}
    Suppose $p^{3}N^{2}\to\infty$ as $N\to\infty$. We have for $N\to\infty$
    \begin{equation*}
        \beta_{\text{eff}}(N)\to\beta.
    \end{equation*}
    \label{convergencebp}
\end{Lem}
\begin{proof}
    Set 
    \begin{equation*}
        F_p(x):=\frac{1}{2}\log \frac{(1-p)+pe^{x}}{(1-p)+pe^{-x}},
    \end{equation*}
    which implies by \eqref{b} 
    \begin{equation}
    F_p\left(\frac{\beta}{2pN}\right)=\frac{\beta_{\text{eff}}}{2N}.\label{eq:lemma241}
    \end{equation}
    We find that
    \begin{equation*}
        F_p^\prime(x)=\frac{pe^{x}}{2 \left(pe^{x} - p + 1\right)} + \frac{pe^{-x}}{2 \left(pe^{-x} - p + 1\right)}~~~\text{and}~~~ F_p^\prime(0)=p.
    \end{equation*}
    We also find that $F_p(x)$ is odd since
    \begin{equation*}
        F_p(-x)=\frac{1}{2}\log \frac{(1-p)+pe^{-x}}{(1-p)+pe^{x}}=\frac{1}{2}\log \left(\frac{(1-p)+pe^{x}}{(1-p)+pe^{-x}}\right)^{-1}=-F_p(x).
    \end{equation*}
    Since $F_p(x)$ is infinitely differentiable at zero, a Taylor expansion around zero thus yields for $x\to0$
    \begin{equation*}
        F_p(x)=px+O\left(x^{3}\right).
    \end{equation*}
    We now find by equation \eqref{eq:lemma241}
    \begin{equation}
        \beta_{\text{eff}}=2NF_p\left(\frac{\beta}{2pN}\right)=\beta+ O\left(\frac{\beta^{3}}{p^{3}N^{2}}\right)\label{ratebetaeff}
    \end{equation}
    and hence
    \begin{equation*}
        \beta_{\text{eff}}(N)\to\beta~~\text{ for }N\to\infty,
    \end{equation*}
    where we used $p^{3}N^{2}\to\infty$ as $N\to\infty$.
\end{proof}
We now apply the so-called Hubbard-Stratonovich transform to the summands in \eqref{annealdpartition2}, which relies on the simple identity
\begin{equation}
e^{\frac{c^{2}}{2d}}
 = \sqrt{\frac{d}{2\pi}}
   \int_{-\infty}^{\infty}
      e^{-\frac{d}{2}s^2 + cs }\operatorname{d}s \label{hubburdstratonovich}
\end{equation}
for any $c,d\in\mathbb{R}$.
We want to use this equation to represent the partition function $Z_{N,p,\beta}(h)$ as an integral. We find by \eqref{annealdpartition2} and \eqref{hubburdstratonovich} setting $c=M_N$ and $d=N/\beta_{\text{eff}}$
\begin{equation}
Z_{N,p,\beta}(h)
=a^{N^{2}}\sqrt{\frac{N}{\beta_{\text{eff}} 2\pi}}\sum_{\sigma\in\left\lbrace-1,1\right\rbrace^N}\int_{-\infty}^\infty \exp\left(-\frac{N}{2\beta_{\text{eff}}}s^{2}+(h+s)M_N(\sigma)\right)\operatorname{d}s.\label{asymptoticspartI}
\end{equation}
Now recall that
\begin{equation*}
    \cosh(x):=\frac{e^{x}+e^{-x}}{2}.
\end{equation*}
We thus find that
\begin{equation}
    \sum_{\sigma\in\left\lbrace -1,1\right\rbrace^{N}}\exp\left((h+s)M_N(\sigma)\right)=2^{N}\cosh^{N}(h+s).\label{asymptoticspartII}
\end{equation}
Now plugging \eqref{asymptoticspartII} into \eqref{asymptoticspartI} yields the following lemma:
\begin{Lem}\label{lem:alternativerep}
    Define
    \begin{equation*}
        \Phi_N(s,h):=\frac{-s^{2}}{2\beta_{\text{eff}}}+\log\left(2\cosh(h+s)\right)~~~\text{ for }s,h\in\mathbb{R}.
    \end{equation*}
    We have the following alternative representation of the annealed partition function:
    \begin{equation}
        Z_{N,p,\beta}(h)=a^{N^{2}}\sqrt{\frac{N}{\beta_{\text{eff}} 2 \pi}}\int_{-\infty}^\infty \exp\left(N\Phi_N(s,h)\right)\operatorname{d}s.\label{asymptoticspartIII}
    \end{equation}
\end{Lem}
We now turn to the proof of Theorem \ref{Statuleviciusthmcw}. 

\section{Statulevi\v{c}ius condition for the annealed dilute Curie-Weiss model}\label{Sec3}
We start this section by providing a short outline of the proof of Theorem \ref{Statuleviciusthmcw}. Throughout this outline, we consider a fixed $h_0\in\mathbb{R}$ and $\beta$ with $0<\beta<1$. Moreover, we assume that $p^{3}N^{2}\to\infty$ as $N\to\infty$.
\begin{enumerate}
    \item The first step consists of using \eqref{representationcumulanderiv} to write for $j\geq3$
    \begin{equation}
        \vert \kappa_j(M_N)\vert =\left\vert N\left.\frac{\operatorname{d}^j}{\operatorname{d}h^j}\psi_{N,p,\beta}(h)\right\vert_{h=h_0}\right\vert. \label{cumulanttopressure}
    \end{equation}
    \item We analytically extend the function $\psi_{N,p,\beta}(h)$ to an open strip $U_N\subset \mathbb{C}$ around the real axis which, as it turns out below, converges to an open strip $U\subset \mathbb{C}$ as $N\to\infty$. We can then find an $R>0$, independent of $N$, such that the closure of the disc $D(h_0,R)$ around $h_0$ with radius $R$ is contained in $U$ and $U_{N}$ for all large enough $N$. Applying Cauchy's estimate \cite[Chapter 2; Corollary 4.3.]{stein2010complex} to the right-hand side of \eqref{cumulanttopressure} yields
    \begin{equation}
        \vert \kappa_j(M_N)\vert \leq N \frac{j!}{R^{j}}\sup_{\vert h-h_0\vert=R}\vert \psi_{N,p,\beta}(h)\vert. \label{cumulantbound}
    \end{equation}
    This is the reason why we need to consider complex external magnetic fields.
    \item To bound the supremum on the right-hand side of \eqref{cumulantbound}, we show that $\psi_{N,p,\beta}(h)$ converges locally uniformly to $f(h)$ on a domain $U$, where $f(h)$ is a holomorphic function on $U$, to be specified more explicitly in the remainder of the present article. For this, we use the saddle-point method from complex and asymptotic analysis while taking special care of the uniformity (see e.g. \cite[Chapter 45.4, Theorem 2]{SidorovFedoryukShabunin1985} for a pointwise version of the saddle-point method). The starting point of this argument is the integral representation \eqref{asymptoticspartIII} of $Z_{N,p,\beta}(h)$  in Lemma \ref{lem:alternativerep}.
    \item We can now use local uniform convergence to conclude that the supremum in \eqref{cumulantbound} is bounded by a constant $C(h_0,R)>0$. We then use the invariance and the homogeneity of the cumulants of higher order to close the argument.
\end{enumerate}
\begin{Rem}
    Since $\psi_{N,p,\beta}(h)$ is only defined on $U_N$, the convergence stated in item 3 of the outline is a priori not well-defined. We will find that the strips $U_N$ converge to $U$ as $N\to\infty$. In that case, every compact $K\subset U$ is eventually contained in $U_N$ for $N$ large enough. We thus consider the expression 
    \begin{equation*}
        \sup_{\vert h-h_0\vert=R}\vert\psi_{N,p,\beta}(h)-f(h)\vert
    \end{equation*}
    only for sufficiently large $N$. We frequently come across this convergence throughout this article. In these cases, local uniform convergence is to be understood in that sense.\label{Remarkconv}
\end{Rem}
\begin{Rem}
    The bound of the Statulevi\v{c}ius condition in Theorem \ref{Statuleviciusthmcw} depends on the positive parameter $R$. In \cite{kabluchko2022lee} the upper bound $\arccos\left(\sqrt{\beta}\right)-\sqrt{\beta(1-\beta)}$ of $R$ has been recognized to be the distance from the origin to the closest zeros of the infinite-volume partition function $\lim_{N\to\infty}Z_{N,1,\beta}(h)$, commonly referred to as Lee-Yang zeros, of the classical Curie-Weiss model.    
    This is due to the fact that the (annealed) partition function $Z_{N,p,\beta}(h)$ of the dilute Curie-Weiss model converges to $\lim_{N\to\infty}Z_{N,1,\beta}(h)$ if $p^{3}N^{2}\to\infty$ as $N\to\infty$. 
    Since we need to ensure that the pressure $\lim_{N\to\infty}\psi_{N,p,\beta}(h)$ is holomorphic on $U$, $\lim_{N\to\infty}Z_{N,1,\beta}(h)$ has no zeros in $U$. Thus $\arccos\left(\sqrt{\beta}\right)-\sqrt{\beta(1-\beta)}$ is going to be the width of the strip $U\subset\mathbb{C}$ around the real axis.
\end{Rem}
\subsection{Proof of Theorem \ref{Statuleviciusthmcw}}
\label{Proofofthm}
The key Proposition for proving Theorem \ref{Statuleviciusthmcw} is the following:
\begin{Prop}
    Let $0<\beta<1$ and suppose that $p^{3}N^{2}\to\infty$ as $N\to\infty$. Set
    \begin{align*}
        U&:= \left\lbrace z\in\mathbb{C}\colon \vert\operatorname{Im}z\vert<\arccos(\sqrt{\beta})-\sqrt{\beta(1-\beta)}\right\rbrace,\\
        U_N&:= \left\lbrace z\in\mathbb{C}\colon \vert\operatorname{Im}z\vert<\arccos(\sqrt{\beta_{\text{eff}}(N)})-\sqrt{\beta_{\text{eff}}(N)(1-\beta_{\text{eff}}(N))}\right\rbrace,\\
        T&:=\left\lbrace z\in\mathbb{C}\colon \vert\operatorname{Im}z\vert\leq\sqrt{\beta(1-\beta)}\right\rbrace\\
        \text{and}\quad\Phi&: T\times U\to\mathbb{C};(s,h)\mapsto\frac{-s^{2}}{2\beta}+\log\left(2\cosh(h+s)\right).
    \end{align*}
    Locally uniformly on $U$, we have 
    \begin{equation}
        \psi_{N,p,\beta}(h)= \Phi(s(h),h) + O\left(\frac{1}{pN}\right), \label{asymptoticeq}
    \end{equation}\label{asymptoticprop}
    where $s\colon U\to T$ is the unique holomorphic function that solves 
    \begin{equation*}
        s(h)=\beta\tanh(h+s(h))
    \end{equation*}
    and $\psi_{N,p,\beta}(h)$ is holomorphic on $U_N$.
\end{Prop}
\begin{Rem}
    \begin{enumerate}
    \item Observe that the function $\Phi(s,h)$ is holomorphic on $T\times U$. Moreover, $s(h)=\beta m(h)$ where $m(h)$ is the solution of the mean-field equation defined in Theorem \ref{Statuleviciusthmcw}.
    \item 
    Note that $\Phi_N(s,h)\to\Phi(s,h)$ for $N\to\infty$ for all $(s,h)\in\mathbb{C}^{2}$ such that $\vert\operatorname{Im}(s+h)\vert< \pi/2$. Moreover, we find that, depending on $\beta_{\text{eff}}$, $\psi_{N,p,\beta}(h)$ is holomorphic on a smaller strip, as soon as $\beta_{\text{eff}}$ is close to $\beta$.
    \end{enumerate}
\end{Rem}
Assuming Proposition \ref{asymptoticprop}, we can now prove Theorem \ref{Statuleviciusthmcw} and Corollary \ref{cor:largedeviation}.
\begin{proof}[Proof of Theorem~\ref{Statuleviciusthmcw}]
Fix $h_0\in\mathbb{R}$ and $R>0$ such that  $\overline{D(h_0,R)}\subset U$, where $\overline{D(h_0,R)}$ denotes the closure of the disc of radius $R$ around $h_0$.
By Proposition \ref{asymptoticprop} and the Weierstrass convergence theorem \cite[Chapter 2, Theorems 5.2 and 5.3]{stein2010complex},
$\psi_{N,p,\beta}(\cdot)\to \Phi(s(\cdot),\cdot)$ locally uniformly on $U$ implies
$\psi_{N,p,\beta}^{(k)}(\cdot)\to (\Phi(s(\cdot),\cdot))^{(k)}$ locally uniformly on $U$
for every fixed $k\in\mathbb{N}$.
We first treat the expectation and the variance of $M_N$.
 Since $s(h)$ solves $\Phi_s(s(h),h)=0$, we have
\begin{equation*}
\frac{\operatorname{d}}{\operatorname{d}h}\Phi(s(h),h)=\Phi_h(s(h),h)=\tanh(h+s(h))=m(h).
\end{equation*}
Since $s(h)=\beta m(h)$ and $m(h)$ is the unique solution of the mean-field equation in Theorem \ref{Statuleviciusthmcw} we have
$m(h)=\tanh(h+\beta m(h))$. Differentiating yields
\begin{equation*}
m'(h)=\frac{1-m(h)^2}{1-\beta(1-m(h)^2)}=: \chi(h).
\end{equation*}
Hence, by equation \eqref{representationcumulanderiv} and Remark \ref{remarkmomentscumulants} we have $\mathbb{E}[M_N]=\kappa_1(M_N)=N\psi_{N,p,\beta}^\prime(h_0)$ and $\mathbb{V}[M_N]=\kappa_2(M_N)=N\psi_{N,p,\beta}^{\prime\prime}(h_0)$,
\begin{equation*}
\frac{\mathbb{E}[M_N]}{N}=\psi'_{N,p,\beta}(h_0)=m(h_0)+o(1),
\qquad
\frac{\mathbb{V}[M_N]}{N}=\psi''_{N,p,\beta}(h_0)=\chi(h_0)+o(1).
\end{equation*}

We now bound the cumulants of order $j\geq3$.
Since $\psi_{N,p,\beta}$ is holomorphic on $U_N$ and $U_N\to U$ as $N\to\infty$ by Lemma \ref{convergencebp}, we can find $N_0\in\mathbb{N}$ such that for all $N\geq N_0$ we have $\overline{D(h_0,R)}\subset U_N$. Cauchy's estimate \cite[Chapter 2; Corollary 4.3.]{stein2010complex} then gives for $j\ge3$ and $N\geq N_0$
\begin{equation*}
|\kappa_j(M_N)|
= N\left|\psi^{(j)}_{N,p,\beta}(h_0)\right|
\le N \frac{j!}{R^j}\sup_{|h-h_0|=R}|\psi_{N,p,\beta}(h)|.
\end{equation*}
By Proposition \ref{asymptoticprop}, on $\partial D(h_0,R)$ we have
$\psi_{N,p,\beta}(h)=\Phi(s(h),h)+O\left(\frac{1}{pN}\right)$ uniformly, hence
\begin{equation*}
\sup_{|h-h_0|=R}|\psi_{N,p,\beta}(h)|
\leq \sup_{|h-h_0|=R}|\Phi(s(h),h)| + C_0
=: C_1,
\end{equation*}
with $C_1<\infty$ independent of $N$. Therefore,
\begin{equation*}
|\kappa_j(M_N)| \leq N\, C_1\, \frac{j!}{R^j}.
\end{equation*}

Now for $j\ge3$ cumulants are shift-invariant and homogeneous, so
\begin{equation*}
\left|\kappa_j\!\left(\frac{M_N-\mathbb{E}[M_N]}{\sqrt{\mathbb{V}[M_N]}}\right)\right|
=\frac{|\kappa_j(M_N)|}{\mathbb{V}[M_N]^{j/2}}.
\end{equation*}
Using $\mathbb{V}[M_N]=N(\chi(h_0)+o(1))$ and $\chi(h_0)>0$
for $\beta<1$, there exists, after possibly increasing $N_0$, a constant $c_*>0$ such that
$\mathbb{V}[M_N]\ge c_* N$ for all $N\ge N_0$. Hence for $N\ge N_0$,
\begin{equation*}
\left|\kappa_j\!\left(\frac{M_N-\mathbb{E}[M_N]}{\sqrt{\mathbb{V}[M_N]}}\right)\right|
\le \frac{N\,C_1\,\frac{j!}{R^j}}{(c_*N)^{j/2}}
= C \,\frac{j!}{(Rc\sqrt{N})^{j-2}},
\end{equation*}
for $c=c(h_0,\beta)=\sqrt{c^*}$ and a constant $C=C(h_0,\beta)$ (and locally uniformly in $h_0$ on compact sets).
This is the Statulevi\v{c}ius condition with $\gamma=0$ and $\Delta=(R/\tilde{C})\sqrt{N}$ with $\tilde{C}:=\max(C/c,1)$.
\end{proof} 
\begin{proof}[Proof of Corollary \ref{cor:largedeviation}]
    By Lemma \ref{lem:representationcumulant} and Proposition \ref{asymptoticprop} we have for every $x\in U$
    \begin{equation*}
        \frac{1}{N}\log\mathbb{E}\left[e^{xN\frac{M_N}{N}}\right]=\psi_{N,p,\beta}(h+x)-\psi_{N,p,\beta}(h)\xrightarrow[]{N\to\infty}\Phi(s(h+x),h+x)-\Phi(s(h),h)=:\Lambda(x)
    \end{equation*}
    locally uniformly on $U$. Note that due to Proposition \ref{asymptoticprop} we have $\Lambda(x)<\infty$ for all $x\in\mathbb{R}\subset U$.  Due to the local uniform convergence and since $\mathbb{R}\subset U$ we also find, again by the Weierstrass convergence theorem \cite[Chapter 2, Theorems 5.2 and 5.3]{stein2010complex}, that $\Lambda(x)$ is differentiable on $\mathbb{R}$ and thus in particular lower semicontinuous on $\mathbb{R}$. The assumptions of the Gärtner-Ellis theorem \cite[Theorem V.6]{hollander2000large} are thus met, which immediately implies Corollary \ref{cor:largedeviation} with speed $N$ and rate function
    \begin{equation*}
        I(x):=\sup_{t\in\mathbb{R}}(xt-\Lambda(t)).
    \end{equation*}
\end{proof}
The remainder of this paper is thus devoted to proving Proposition \ref{asymptoticprop}.

\subsection{Uniform asymptotics of the finite-volume pressure -- Proof of Proposition \ref{asymptoticprop}} \label{Proofofprop}
We first derive the asymptotics of the integral $\int_{-\infty}^\infty e^{N\Phi_N(s,h)}\operatorname{d}s$ as $N\to\infty$, where $\Phi_N$ was defined in Lemma \ref{lem:alternativerep}. For our technique, we need to consider $h\in \mathbb{C}$. In this case, $\Phi_N(s,h)$ is complex-valued and the usual Laplace method cannot be applied. Instead, we resort to the saddle-point method. It consists of finding the saddle-point $s_N(h)$ of the function $s\mapsto\Phi_N(s,h)$, which depends on $N$ and $h$, and shifting the path of integration through it. The value of the integral does not change after shifting the contour by Cauchy's theorem (see e.g. \cite[Chapter 2; Corollary 1.2]{stein2010complex}). In particular, we want to find a contour $\gamma_N^\prime\subset\mathbb{C}$ that is homotopic to $\left[-\infty,\infty\right]$ such that the region between the two contours contains no singularities and such that for every $h\in U$ (see e.g. \cite[Chapter 45.4, Theorem 2]{SidorovFedoryukShabunin1985} for a pointwise version of the saddle-point method):
\begin{enumerate}
    \item $\max_{s\in\gamma_N^\prime}\operatorname{Re}\Phi_N(s,h)$ is attained only at a point $s_N(h)$ that is an interior point of $\gamma^\prime_N$ and is a simple saddle-point (i.e. $(\Phi_N)_{s}(s_N(h),h)=0$ and  $(\Phi_N)_{ss}(s_N(h),h)\neq 0$).
    \item The function $\operatorname{Re}(\Phi_N)(s,h)$ is a concave function on $\gamma_N^\prime$ (i.e. $\operatorname{Re}(\Phi_N)_{ss}(s,h)<0$ on $\gamma^\prime_N$).
\end{enumerate}
As item 1 suggests, we need to find a function $s_N(h)$ that solves the so-called saddle-point equation $(\Phi_N)_s(s_N(h),h)=0$ for every $h$ in an appropriate domain. In our case, we are able to find a holomorphic $s_N(h)$ that solves this equation for any $h\in U_N$.
By Lemma \ref{convergencebp} we find that
\begin{equation*}
    U_N\to U~~~\text{ as }N\to\infty,
\end{equation*}
 where $U_N$ and $U$ were defined in Proposition \ref{asymptoticprop}.
Lemma \ref{welldefu1} and Lemma \ref{sonD} ensure that such a holomorphic function $s_N(h)$ satisfying the conditions in item 1 exists and is unique. Having established this, we find that $\gamma_N^\prime(h)=\left\lbrace z\in\mathbb{C}\colon \operatorname{Im}z=\operatorname{Im}s_N(h)\right\rbrace$ is a suitable new contour with respect to item 2, which we verify in Lemma \ref{welldefsecderiv}. Lemma \ref{shiftintegration} then confirms that the shift of the contour is valid.

Next we derive the saddle-point $s_N(h)$ of $\Phi_N(s,h)$. In our case, the function $s_N(h)$ is defined implicitly. Assuming that the derivative exists, the saddle-point equation is given by
\begin{equation}
    G_N(s,h):=(\Phi_N)_s(s,h)=\frac{-s}{\beta_{\text{eff}}}+\tanh(h+s)\stackrel{!}{=}0.\label{saddlepointequation}
\end{equation}
We will use the holomorphic implicit function theorem to show that there exists a unique holomorphic solution $s_N(h)$ on $U_N$ such that \eqref{saddlepointequation} holds. It is immediate that the solution $s_N(h)$ also depends on $\beta$ as $\beta_{\text{eff}}$ also depends on $\beta$. For simplicity we again leave out this dependence in the notation of $s_N(h)$. We first check whether our setup fulfills all the conditions of the implicit function theorem.

We start with the following two technical Lemmas:
\begin{Lem}
    Fix $0\leq y<\pi/2$ and set
    \begin{equation*}
        \mathcal{A}:=\left\lbrace x\in(0,1)\colon y<\arccos(\sqrt{x})-\sqrt{x(1-x)}\right\rbrace.
    \end{equation*}
    Then there exists a unique continuously differentiable function $t\colon\mathcal{A}\to\left[0,\frac{1}{2}\right]$ such that for all $\beta\in\mathcal{A}$ we find that $t(\beta)$ is the unique solution of
    \begin{equation}
        t(\beta)=\beta\tan(y+t(\beta)) \label{eqlemmafixpointtan}
    \end{equation}
    in the interval $\mathcal{B}:=[0,\sqrt{\beta(1-\beta)}]$. Moreover, the function $t$ is strictly increasing.\label{lemmafixpointtan}
\end{Lem}
Before proving the lemma, observe that we have for $z=z_1+iz_2\in\mathbb{C}$ with $z_2\in[0,\pi/2]$ by a standard identity
\begin{eqnarray}
\vert\operatorname{Im}\tanh(z_1+iz_2)\vert&=&\frac{\sin(2z_2)}{\cosh(2z_1)+\cos(2z_2)}\nonumber\\
    &\leq&\frac{\sin(2z_2)}{1+\cos(2z_2)}=\tan(z_2). \label{standardine}
\end{eqnarray}
\begin{proof}[Proof of Lemma \ref{lemmafixpointtan}]
    We want to invoke the implicit function theorem. For that, consider the function $F:\mathcal{A}\times\mathcal{B}\to \mathbb{R}$ defined as
    \begin{equation*}
        F(\beta^\prime,t):=t-\beta^\prime\tan(y+t).
    \end{equation*}
    Fix some $0\leq y<\pi/2$ and choose an arbitrary $\beta\in\mathcal{A}$. We first show that there exists a unique solution $t_0\in[0,\sqrt{\beta(1-\beta)}]$ such that $F(\beta,t_0)=0$. To prove this claim, note that for the function $\mathcal{T}(t):=\beta\tan(y+t)$ we have on $\mathcal{B}$
    \begin{equation*}
        0<\mathcal{T}^{\prime}(t)=\frac{\beta}{\cos^{2}(y+t)}\leq\frac{\beta}{\cos^{2}(y+\sqrt{\beta(1-\beta)})}<\frac{\beta}{\cos^{2}(\arccos{\sqrt{\beta}})}=1.
    \end{equation*}
    Moreover, since $\mathcal{T}$ is strictly increasing on $\mathcal{B}$ we have
    \begin{eqnarray*}
        \mathcal{T}\left(\left[0,\sqrt{\beta(1-\beta)}\right]\right)&\subset& \left[0,\mathcal{T}\left(\sqrt{\beta(1-\beta)}\right)\right]\subset \left[0,\mathcal{T}\left(\arccos(\sqrt{\beta})\right)\right]\\
        &=&\left[0,\beta\tan(\arccos(\sqrt{\beta}))\right]=\left[0,\sqrt{\beta(1-\beta)}\right],
    \end{eqnarray*}
    which proves the existence, the uniqueness and the fact that $t(\beta)\in\mathcal{B}$ by the Banach fixed point theorem. To prove the continuous differentiability, note that at $(\beta,t_0)$ we have $F(\beta,t_0)=0$. For the partial derivative of $F$ with respect to $t$, we also have that
    \begin{equation*}
        \partial_tF(\beta,t_0)=1-\frac{\beta}{\cos^{2}(y+t_0)}>1-\frac{\beta}{\cos^{2}(\arccos(\sqrt{\beta}))}=0.
    \end{equation*}
    By the implicit function theorem there are thus open sets $\beta\in  \mathcal{U}\subset \mathcal{A}$ and $t_0\in V\subset \mathcal{B}$ and a unique continuously differentiable function $t\colon \mathcal{U}\to V$ such that $F(\beta^\prime,t(\beta^\prime))=0$. Moreover, $t$ is strictly increasing on $\mathcal{U}$ since by the implicit function theorem we have for the derivative of $t(\beta^{\prime})$ 
    \begin{equation*}
        t^\prime(\beta^\prime)=-\frac{\partial_{\beta^{\prime}}F(\beta^\prime,t(\beta^\prime))}{\partial_tF(\beta^\prime,t(\beta^\prime))}=\frac{\tan(y+t(\beta^\prime))}{1-\frac{\beta^\prime}{\cos^{2}(y+t(\beta^\prime))}}>0.
    \end{equation*}
    This completes the proof, since the choice of $\beta$ was arbitrary.
\end{proof}
For the proof of the next lemma, set 
\begin{equation}
    \alpha_N:=\arccos\left(\sqrt{\beta_{\text{eff}}}\right)~~~~\text{and}~~~~\delta_N:=\sqrt{\beta_{\text{eff}}\left(1-\beta_{\text{eff}}\right)}.\label{alphadelta}
\end{equation}
Recall that, like $\beta_{\text{eff}}$, $\alpha_N$ and $\delta_N$ depend on $\beta$.
We find by Lemma \ref{convergencebp}
\begin{equation*}
    \alpha_N\to \alpha:=\arccos(\sqrt{\beta})~~~\text{ and }~~~\delta_N\to\delta:=\sqrt{\beta(1-\beta)}~~~\text{as }N\to\infty.\label{defalpha}
\end{equation*}
\begin{Lem}
    Assume $0<\beta<1$ and fix any $h\in U$, where $U$ was defined in Proposition \ref{asymptoticprop}. Consider for $s,h\in\mathbb{C}$ the saddle-point equation
    \begin{equation}
        s=\beta_{\text{eff}}\tanh(h+s).\label{fixpoint}
    \end{equation}
    Then there exists $N_0\in\mathbb{N}$ such that for all $N\geq N_0$ there exists a unique solution $s_N(h)$ of \eqref{fixpoint} and it satisfies
    \begin{equation*}
        \vert\operatorname{Im}s_N(h)\vert\leq\delta_N,
    \end{equation*}
    for $\delta_N$ given in \eqref{alphadelta}.
    \label{welldefu1}
\end{Lem}
\begin{Rem}
    The existence of a unique solution of \eqref{fixpoint} was previously established for $h\in\mathbb{R}$ and $0<\beta_{\text{eff}}=\beta<1$ (see e.g. \cite{friedli2017statistical}). Lemma \ref{welldefu1} provides the existence of a unique solution in a complex strip around the real axis which, to the best of our knowledge, has not been established previously.
\end{Rem}
\begin{proof}[Proof of Lemma \ref{welldefu1}]
    Fix $h\in U$. Set $y:=\vert\operatorname{Im}h\vert$. By Lemma \ref{convergencebp} we can choose $N_0\in\mathbb{N}$ such that $0<\beta_{\text{eff}}< 1$ and $h\in U_N$ for all $N\geq N_0$. Now fixing such an $N_0$, we choose $t(\beta_{\text{eff}})$ according to \eqref{eqlemmafixpointtan} in Lemma \ref{lemmafixpointtan}. We set
    \begin{equation*}
        \overline{S}^\prime:=\left\lbrace s\in\mathbb{C}\colon \vert\operatorname{Im}s\vert \leq t(\beta_{\text{eff}})\right\rbrace.
    \end{equation*}
    We again want to use the Banach fixed-point theorem for \eqref{fixpoint} and we thus set 
    \begin{equation*}
        \mathcal{T}_{N,h}(s):=\beta_{\text{eff}}\tanh(h+s)
    \end{equation*}
    on $\overline{S}^\prime$. We have by \eqref{eqlemmafixpointtan} and \eqref{standardine} for $s\in\overline{S}^\prime$
    \begin{eqnarray*}
        \beta_{\text{eff}}\vert\operatorname{Im}(\tanh(h+s))\vert&\leq& \beta_{\text{eff}}\tan(\vert\operatorname{Im}(h+s)\vert)\leq\beta_{\text{eff}}\tan(y+t(\beta_{\text{eff}}))=t(\beta_{\text{eff}}),
    \end{eqnarray*}
    which shows that $\mathcal{T}_{N,h}\left(\overline{S}^\prime\right)\subset\overline{S}^\prime$. Moreover, for $s\in \overline{S}^\prime$ we find
    \begin{equation*}
         \vert \mathcal{T}^\prime_{N,h}(s)\vert\leq \frac{\beta_{\text{eff}}}{\vert\cosh^{2}(h+s)\vert}\leq\frac{\beta_{\text{eff}}}{\cos^{2}(y+t(\beta_{\text{eff}}))}<\frac{\beta_{\text{eff}}}{\cos^{2}(\alpha_N)}=1,
    \end{equation*}
    where the strict inequality holds since $y=\vert\operatorname{Im}h\vert<\alpha_N-\delta_N$ due to $h\in U_N$, showing that $\mathcal{T}_{N,h}$ is a contraction.
    Thus there exists for every $N\geq N_0$ a unique fixed point $s_N(h)\in\overline{S}^\prime$ such that 
    \begin{equation*}
        s_N(h)=\beta_{\text{eff}}\tanh(h+s_N(h)),
    \end{equation*}
    with the property that $\vert\operatorname{Im}s_N(h)\vert\leq\delta_N$, since $t(\beta_{\text{eff}})\leq \delta_N$ by Lemma \ref{lemmafixpointtan}.
\end{proof}
We now check that $(G_N)_s(s,h)$ is nonzero on $U_N$ to justify the implicit function theorem.
\begin{Lem}
    Suppose $0<\beta<1$ and $p^{3}N^{2}\to\infty$ as $N\to\infty$. 
    Then there exists $N_0\in\mathbb{N}$ such that for all $N\geq N_0$ and all $h\in\mathbb{C}$ satisfying 
    \begin{equation*}
        \left\vert\operatorname{Im}(h)\right\vert<\left\vert\arccos\left(\sqrt{\beta_{\text{eff}}}\right)-\sqrt{\beta_{\text{eff}}(1-\beta_{\text{eff}})}\right\vert,
    \end{equation*}
    we have that $H_N(s_N(h),h):=(\Phi_N)_{ss}(s_N(h),h)=-\frac{1}{\beta_{\text{eff}}}+\frac{1}{\cosh^{2}(s_N(h)+h)}$ is nonzero, where $s_N(h)$ is the unique solution of \eqref{fixpoint}. Moreover, there exists a $c\in\mathbb{R}$ such that $\operatorname{Re}H_N(s_N(h),h)<c<0$ locally uniformly on $U$ for large enough $N$.
    \label{welldefsecderiv}
\end{Lem}
\begin{proof}
    Set $s+h=:z_1+iz_2=:z$.  We write $H_N(z)$ in short for $H_N(s,h)$. By Lemma \ref{convergencebp} we can take $N_0\in\mathbb{N}$ such that for all $N\geq N_0$ we have that $0<\beta_{\text{eff}}\leq1$. Now fix some $N\geq N_0$ and take $h\in\mathbb{C}$ satisfying $\vert\operatorname{Im}(h)\vert< \vert\arccos(\sqrt{\beta_{\text{eff}}})-\sqrt{\beta_{\text{eff}}(1-\beta_{\text{eff}})}\vert$.  Assume, for the sake of contradiction, that $H_N(z)=0$. This is true if and only if
    \begin{equation*}
        \beta_{\text{eff}}=\cosh^{2}(z).
    \end{equation*}
    Now the condition $\cosh(z)=\sqrt{\beta_{\text{eff}}}\leq1$ forces $z$ to be purely imaginary, i.e. $z=iz_2$ for $z_2\in\mathbb{R}$.
    In this case, we then also find for $k\in\mathbb{Z}$
    \begin{equation*}
        \beta_{\text{eff}}=\cosh^{2}(iz_2)=\cos^{2}(z_2+2k\pi).
    \end{equation*}
    For $k=0$ we find
    \begin{equation*}
        z_2=\pm\arccos\sqrt{\beta_{\text{eff}}}.\label{zerosz2}
    \end{equation*}
     We then get for $\alpha_N:=\arccos{\sqrt{\beta_{\text{eff}}}}$ as defined in \eqref{alphadelta} and due to the relation $\tanh(ix)=i\tan(x)$ for $x\in\mathbb{R}$
    \begin{equation*}
        \pm h=\pm i\alpha_N\mp \beta_{\text{eff}}\tanh( i\alpha_N)=i(\pm\alpha_N\mp \beta_{\text{eff}}\tan(\alpha_N))=i\left(\pm \alpha_N\mp\sqrt{\beta_{\text{eff}}\left(1-\beta_{\text{eff}}\right)}\right),
    \end{equation*}
    which contradicts the assumption that $\vert\operatorname{Im}h\vert<\vert\arccos(\sqrt{\beta_{\text{eff}}})-\sqrt{\beta_{\text{eff}}(1-\beta_{\text{eff}})}\vert$. This proves the first part of the statement. To prove the second part, let $K\subset U$ be compact ($U$ was defined in Proposition \ref{asymptoticprop}). As before by Lemma \ref{convergencebp}, we can choose $N_0\in\mathbb{N}$ such that $K\subset U_N$ for all $N\geq N_0$. We have that
    \begin{equation*}
        \operatorname{Re}H_N(s_N(h),h)=-\frac{1}{\beta_{\text{eff}}}+\operatorname{Re}\frac{1}{\cosh^{2}(s_N(h)+h)}
    \end{equation*}
    and
    \begin{equation*}
        \operatorname{Re}\frac{1}{\cosh^{2}(s_N(h)+h)}\leq\frac{1}{\vert\cosh^{2}(s_N(h)+h)\vert}\leq \frac{1}{\cos^{2}(\operatorname{Im}(s_N(h)+h))}.
    \end{equation*}
    We claim that on $K$ we have $\cos(\vert\operatorname{Im}(s_N(h)+h)\vert)>c>\sqrt{\beta_{\text{eff}}}$. Since $K\subset U_N$, we can find $\eta>0$ such that for all $h\in K$ we have $\vert\operatorname{Im}h\vert\leq \alpha_N-\delta_N-\eta$. In this case, it also holds that $\vert\operatorname{Im}s_N(h)\vert\leq \delta_N$. We thus have
    \begin{equation*}
        \cos(\vert\operatorname{Im}(s_N(h)+h)\vert)\geq\cos(\alpha_N-\eta).
    \end{equation*}
    Observe that $\sqrt{\beta_{\text{eff}}}=\cos(\alpha_N)$. We have
    \begin{equation*}
        \cos(\alpha_N-\eta)-\cos(\alpha_N)=-\int_{0}^{-\eta}\sin(\alpha_N+t)\operatorname{d}t>-\int_{0}^{-\eta}\sin(\alpha_N)\operatorname{d}t=\eta\sqrt{1-\beta_{\text{eff}}}.
    \end{equation*}
    After possibly increasing $N_0$ we can ensure that $\beta_{\text{eff}}<c^\prime<1$ for some $c^\prime>0$. This completes the proof.
\end{proof}
By Lemma \ref{welldefsecderiv} the domain of $s_N(h)$ is 
\begin{equation*}
    U_N=\left\lbrace z\in\mathbb{C}\colon \vert\operatorname{Im}z\vert<\alpha_N-\delta_N\right\rbrace,\label{defUN}
\end{equation*}
as defined in Proposition \ref{asymptoticprop}.
Next, we prove that $s_N(h)$ is holomorphic on $U_N$.
\begin{Lem}
Fix $0<\beta<1$ and suppose that $p^{3}N^{2}\to\infty$ as $N\to\infty$. There exists a unique holomorphic function 
$s_N\colon U_N\to T_N:=\{z\in\mathbb C:\ |\operatorname{Im} z|\leq \delta_N\}$ 
that solves the saddle-point equation \eqref{saddlepointequation}.\label{sonD}
\end{Lem}

\begin{proof}
Fix $N$ sufficiently large so that Lemma \ref{welldefu1} and Lemma \ref{welldefsecderiv} apply on $U_N$.
Recall that
\begin{equation*}
G_N(s,h)=(\Phi_N)_s(s,h)=-\frac{s}{\beta_{\mathrm{eff}}}+\tanh(h+s).
\end{equation*}
Then $G_N$ is holomorphic in $(s,h)$ on the domain
\begin{equation*}
W:=\{(s,h)\in\mathbb C^2:\ |\operatorname{Im}(h+s)|<\tfrac{\pi}{2}\},
\end{equation*}
since $\tanh$ is holomorphic on the strip $|\operatorname{Im} z|<\pi/2$.
Moreover, by Lemma \ref{welldefu1}, for every $h\in U_N$ there exists a unique $s_N(h)\in T_N$
such that $G_N(s_N(h),h)=0$.

We claim that for every $h\in U_N$ we have $(G_N)_s(s_N(h),h)\neq 0$.
Indeed,
\begin{equation*}
(G_N)_s(s,h)= (\Phi_N)_{ss}(s,h)=-\frac{1}{\beta_{\mathrm{eff}}}+\frac{1}{\cosh^2(h+s)}=H_N(s,h),
\end{equation*}
so Lemma~\ref{welldefsecderiv} yields $H_N(s_N(h),h)\neq 0$ for all $h\in U_N$.

Now fix $h_0\in U_N$ and set $s_0:=s_N(h_0)$. We have $G_N(s_0,h_0)=0$ and
$(G_N)_s(s_0,h_0)\neq 0$. Hence, by the holomorphic implicit function theorem,
there exist open discs $D_{h_0}\subset U_N$ with $h_0\in D_{h_0}$, $B_{s_0}\subset T_N$ with
$s_0\in B_{s_0}$, and a unique holomorphic map
\begin{equation*}
g_{h_0}:D_{h_0}\to B_{s_0}
\end{equation*}
such that $g_{h_0}(h_0)=s_0$ and
\begin{equation*}
G_N(g_{h_0}(h),h)=0~~~\text{for all}~~~~~h\in D_{h_0}.
\end{equation*}
In particular, $g_{h_0}(h)\in T_N$ for $h$ sufficiently close to $h_0$ (shrinking $D_{h_0}$ if necessary),
since $s_0\in T_N$ and $T_N$ is closed.

The family $\{D_{h_0}\}_{h_0\in U_N}$ is an open cover of $U_N$.
Therefore we can define a global function $s_N:U_N\to T_N$ by setting, for any $h\in U_N$,
\begin{equation*}
s_N(h):=g_{h_0}(h)\quad\text{for any }h_0\in U_N\text{ with }h\in D_{h_0}.
\end{equation*}
This is well-defined because the $g_{h_0}$ coincide on overlaps.
Let $h_0,h_1\in U_N$ and $h\in D_{h_0}\cap D_{h_1}$. Then
\begin{equation*}
G_N(g_{h_0}(h),h)=0 ~~~~\text{ and }~~~~ G_N(g_{h_1}(h),h)=0.
\end{equation*}
By Lemma \ref{welldefu1} the solution of $G_N(s,h)=0$ in $T_N$ is unique, hence
$g_{h_0}(h)=g_{h_1}(h)$ for all $h\in D_{h_0}\cap D_{h_1}$.

Moreover, for each $h\in U_N$ we have $s_N=g_{h_0}$ in a neighborhood of $h$, hence $s_N$ is holomorphic on $U_N$.
By construction $G_N(s_N(h),h)=0$ for all $h\in U_N$, i.e.\ $(\Phi_N)_s(s_N(h),h)=0$. Uniqueness follows by Lemma \ref{welldefu1}.
\end{proof}

\begin{Lem}
    Suppose $0<\beta<1$ and $p^{3}N^{2}\to\infty$ as $N\to\infty$. 
    Then $s_N(h)$ converges locally uniformly to $s(h)$ on $U$, i.e. for every compact $K\subset U$
    \begin{equation}
        \sup_{h\in K}\vert s_N(h)-s(h)\vert\to 0\text{ as } N\to\infty.\label{equniform}
    \end{equation}
    Here, $s(h)$ is the unique holomorphic solution $U\to T$ of \eqref{fixpoint} for $\beta_\text{eff}=\beta$.
    \label{uniforms}
\end{Lem}
\begin{Rem}
    The local uniform convergence in Lemma \ref{uniforms} is to be understood as described in Remark \ref{Remarkconv}.
\end{Rem}
\begin{proof}
    Fix a compact set $h\in K\subset U$. We know by Lemma \ref{lemmafixpointtan} that there is a unique continuously differentiable function $t\colon \mathcal{A}\to \left[0,\frac{1}{2}\right]$ solving
    \begin{equation*}
        t(\beta^\prime)=\beta^\prime\tan(\vert\operatorname{Im}h\vert+t(\beta^\prime))
    \end{equation*}
    where $\mathcal{A}= \left\lbrace x\in(0,1)\colon \vert\operatorname{Im}h\vert<\arccos(\sqrt{x})-\sqrt{x(1-x)}\right\rbrace$. We now choose $N_0\in\mathbb{N}$ large enough such that $K\subset U_N$, $0<\beta_{\text{eff}}<1$ and $\beta_{\text{eff}}\in\mathcal{A}$ for all $N\geq N_0$. We now set
    \begin{equation*}
        \overline\beta:=\sup_{N\geq N_0}\left\lbrace \beta_{\text{eff}}\right\rbrace.
    \end{equation*}
    Then $\beta\leq \overline{\beta}<\infty$ since $\beta_{\text{eff}}\to\beta$ as $N\to\infty$. Moreover, define the set
    \begin{equation*}
        \overline{S}^*:=\left\lbrace z\in\mathbb{C}\colon \vert\operatorname{Im}z\vert\leq t(\overline{\beta})\right\rbrace.
    \end{equation*}
    Since $t$ is strictly increasing on $\mathcal{A}$ we have that $t(\beta_{\text{eff}}),t(\beta)\leq t(\overline{\beta})$ for all $N\geq N_0$. Following the proof of Lemma \ref{welldefu1} we then also have $s_N(h),s(h)\in \overline{S}^*$. We now show that the operator
    \begin{equation*}
        \mathcal{T}_{N,h}(s):=\beta_{\text{eff}}\tanh(h+s)
    \end{equation*}
    from the proof of Lemma \ref{welldefu1} is a contraction on $\overline{S}^*$. Indeed we have for $s\in\overline{S}^*$ by Lemma \ref{lemmafixpointtan} for all $N\geq N_0$
    \begin{equation*}
        \vert\operatorname{Im}\mathcal{T}_{N,h}(s)\vert=\beta_{\text{eff}}\vert\operatorname{Im}\tanh(h+s)\vert\leq \beta_{\text{eff}}\tan(\vert\operatorname{Im}(h+s)\vert)\leq \overline{\beta}\tan(\vert\operatorname{Im}h\vert+t(\overline{\beta}))=t(\overline{\beta})
    \end{equation*}
    and
    \begin{equation*}
        \vert \mathcal{T}^\prime_{N,h}(s)\vert\leq \frac{\beta_{\text{eff}}}{\vert \cosh^{2}(h+s)\vert}\leq\frac{\overline{\beta}}{\cos^{2}\left(\vert\operatorname{Im}h\vert+t\left(\overline{\beta}\right)\right)}=:q_h<\frac{\overline{\beta}}{\cos^{2}\left(\arccos\left(\sqrt{\overline{\beta}}\right)\right)}=1,
    \end{equation*}
    since $t(\overline{\beta})\leq \sqrt{\overline{\beta}(1-\overline{\beta})}$ by Lemma \ref{lemmafixpointtan}.
    Note that since $K$ is compact and $\cos$ is holomorphic we even have $q:=\sup_{h\in K}q_h<1$. 
    Now let 
    \begin{equation*}
        \mathcal{T}_h(s):=\beta\tanh(h+s).
    \end{equation*}
    We know by Lemma \ref{welldefu1} that $s(h)=\mathcal{T}_h(s(h))$ and $s_N(h)=\mathcal{T}_{N,h}(s_N(h))$. We now have by the triangle inequality
    \begin{equation}
        \vert s_N(h)-s(h)\vert\leq \vert \mathcal{T}_{N,h}(s_N(h))-\mathcal{T}_{N,h}(s(h))\vert+\vert \mathcal{T}_{N,h}(s(h))-\mathcal{T}_h(s(h))\vert.\label{pw1}
    \end{equation}
    For the first term on the right-hand side, we know by the considerations above that
    \begin{equation}
        \vert \mathcal{T}_{N,h}(s_N(h))-\mathcal{T}_{N,h}(s(h))\vert\leq q_h\vert s_N(h)-s(h)\vert,\label{pw2}
    \end{equation}
    and for the second term
    \begin{equation}
        \vert \mathcal{T}_{N,h}(s(h))-\mathcal{T}_h(s(h))\vert=\vert\beta_{\text{eff}}-\beta\vert\vert\tanh(h+s(h))\vert. \label{pw3}
    \end{equation}
    Now set $C:=\sup_{h\in K}\tanh(\vert\operatorname{Im}(h+s)\vert)$. We have that $C<\infty$ since $\vert\operatorname{Im}(h+s)\vert<\arccos{\sqrt{\overline{\beta}}}$. Putting \eqref{pw2} and \eqref{pw3} into \eqref{pw1} we find
    \begin{equation*}
        \sup_{h\in K}\vert s_N(h)-s(h)\vert\leq \vert\beta_{\text{eff}}-\beta\vert\sup_{h\in K}\frac{\vert\tanh(h+s(h))\vert}{1-q_h}\leq \vert\beta_{\text{eff}}-\beta\vert\frac{C}{1-q},
    \end{equation*}
    where by Lemma \ref{convergencebp} the right-hand side converges to $0$ as $N\to\infty$, proving the statement.
\end{proof}

To apply the saddle-point method, we need to shift the contour of integration in \eqref{asymptoticspartIII} to a contour $\gamma^\prime_N$ that passes through the saddle-point $s_N(h)$ for every $h\in U$. The shift of the contour does not change the value of the integral. This is justified by the following Lemma:
\begin{Lem}
    Let $0<\beta<1$ and suppose that $p^{3}N^{2}\to\infty$ as $N\to\infty$. Let $s_N(h)$ be the unique holomorphic function $U_N\to T_N$ that solves the saddle-point equation \eqref{saddlepointequation}, where $U_N$ was defined in Proposition \ref{asymptoticprop} and $T_N$ was defined in Lemma \ref{sonD}.
    Further define for $h\in U_N$
    \begin{equation*}
        \gamma_N^\prime(h)=\gamma_N^\prime:=\left\lbrace z\in\mathbb{C}\colon \operatorname{Im}z=\operatorname{Im}s_N(h)\right\rbrace \subset T_N.
    \end{equation*}
    We have that 
    \begin{equation*}
        \int_{-\infty}^\infty \exp(N\Phi_N(s,h))\operatorname{d}s=\int_{\gamma_N^\prime}\exp(N\Phi_N(s,h))\operatorname{d}s,~~~\forall N\in\mathbb{N}.
    \end{equation*}\label{shiftintegration}
\end{Lem}
\begin{proof}
    Suppose without loss of generality that $\operatorname{Im}s_N(h)\geq0$. We define the counterclockwise oriented rectangular closed curve by
    \begin{equation*}
        \rho^{R}=\rho_1^{R}\cup\rho_2^{R}\cup\rho_3^{R}\cup\rho_4^{R}
    \end{equation*}
    where $\rho_{i}^{R}\subset T_N$ for $i=1,2,3,4$ are curves, each being parametrized by $\xi_i^{R}\colon [0,1]\to\mathbb{C}$ defined as
    \begin{eqnarray*}
        \xi_1^{R}(t)&:=&-R+2Rt\\
        \xi_{2}^{R}(t)&:=&R+it\operatorname{Im}s_N(h)\\
        \xi_{3}^{R}(t)&:=&R-2Rt+i\operatorname{Im}s_N(h)\\
        \xi_{4}^{R}(t)&:=&-R+i(1-t)\operatorname{Im}s_N(h).
    \end{eqnarray*}
    Since $\rho^{R}$ is a closed curve and $T_N$ is simply connected, we have by \cite[Chapter 2; Corollary 1.2.]{stein2010complex} 
    \begin{equation*}
        \int_{\rho^{R}} \exp(N\Phi_N(s,h))\operatorname{d}s=0.
    \end{equation*} 
    We now show that 
    \begin{equation*}
        \lim_{R\to\infty}\int_{\rho_2^{R}}\exp(N\Phi_N(s,h))\operatorname{d}s=0.
    \end{equation*}
    The limit $\lim_{R\to\infty}\int_{\rho_4^{R}}\exp(N\Phi_N(s,h))\operatorname{d}s=0$ is treated similarly.\\
    We find by \eqref{asymptoticspartI} that
    \begin{equation*}
        \int_{\rho_2^{R}}\exp(N\Phi_N(s,h))\operatorname{d}s=\sum_{\sigma\in\left\lbrace -1,1\right\rbrace^{N}}\int_{\rho_2^{R}}\exp\left(-\frac{N}{2\beta_{\text{eff}}}s^{2}+(h+s)M_N(\sigma)\right)\operatorname{d}s.
    \end{equation*}
    For $s=s_1+is_2\in\mathbb{C}$ and $h=h_1+ih_2\in U_N$ we calculate for some fixed $\sigma\in\left\lbrace -1,1\right\rbrace^{N}$
    \begin{equation*}
        \operatorname{Re}\left(-\frac{N}{2\beta_{\text{eff}}}s^{2}+(h+s)M_N(\sigma)\right)=-\frac{N}{2\beta_{\text{eff}}}(s_1^{2}-s_2^{2})+(h_1+s_1)M_N(\sigma).
    \end{equation*}
    We thus have for any $\sigma\in\left\lbrace -1,1\right\rbrace^{N}$
    \begin{eqnarray*}
        \left\vert\int_{\rho_2^{R}}e^{-\frac{N}{2\beta_{\text{eff}}}s^{2}+(h+s)M_N(\sigma)}\operatorname{d}s\right\vert&\leq& \int_{\rho_2^{R}}e^{-\frac{N}{2\beta_{\text{eff}}}(s_1^{2}-s_2^{2})+(h_1+s_1)M_N(\sigma)}\operatorname{d}s\\
        &=&\int_{0}^{\operatorname{Im}s_N(h)}e^{-\frac{N}{2\beta_{\text{eff}}}(R^{2}-s_2^{2})+(h_1+R)M_N(\sigma)}\operatorname{d}s_2\\
        &=&e^{-\frac{NR^{2}}{2\beta_{\text{eff}}}+(h_1+R)M_N(\sigma)}\int_0^{\operatorname{Im}s_N(h)}e^{s_2^{2}\frac{N}{2\beta_\text{eff}}}\operatorname{d}s_2\\
        &\leq&C^\prime e^{-\frac{NR^{2}}{2\beta_{\text{eff}}}+(h_1+R)M_N(\sigma)}
    \end{eqnarray*}
    where $C^\prime$ is a finite constant independent of $R$ that bounds the integral in the third line. 
    We now have
    \begin{align*}
        \left\vert\sum_{\sigma\in\left\lbrace -1,1\right\rbrace^{N}}\int_{\rho_2^{R}}e^{-\frac{N}{2\beta_{\text{eff}}}s^{2}+(h+s)M_N(\sigma)}\operatorname{d}s\right\vert&\leq C^\prime \sum_{\sigma\in\left\lbrace -1,1\right\rbrace^{N}} e^{-\frac{NR^{2}}{2\beta_{\text{eff}}}+(h_1+R)M_N(\sigma)}\\
        &=C^\prime e^{-\frac{NR^{2}}{2\beta_{\text{eff}}}+N\log\left(2\cosh(h_1+R)\right)}\xrightarrow[]{R\to\infty} 0
    \end{align*}
    Therefore 
    \begin{eqnarray*}   
    \int_{-\infty}^\infty\exp\left(N\Phi_N(s,h)\right)\operatorname{d}s&=&\lim_{R\to\infty}\int_{\rho_1^{R}}\exp\left(N\Phi_N(s,h)\right)\operatorname{d}s\\
    =-\lim_{R\to\infty}\int_{\rho_3^{R}}\exp\left(N\Phi_N(s,h)\right)\operatorname{d}s&=&\int_{\gamma_N^\prime}\exp\left(N\Phi_N(s,h)\right)\operatorname{d}s.
    \end{eqnarray*}
\end{proof}
With Lemma \ref{shiftintegration} we are now in the position to derive the asymptotic behavior of the partition function $Z_{N,p,\beta}(h)$ for $h\in U$. 
\begin{Lem}
    Fix $0<\beta <1$ and suppose $p^{3}N^{2}\to\infty$ as $N\to\infty$. Let $s_N\colon U_N\to T_N$ be the unique holomorphic function that solves \eqref{saddlepointequation} for $\beta_{\text{eff}}$. We have
    \begin{equation}
        Z_{N,p,\beta}(h)=a^{N^{2}}\sqrt{-\frac{1}{\beta_{\text{eff}}(\Phi_N)_{ss}(s_N(h),h)}}e^{N\Phi_N(s_N(h),h)}\left(1+o\left(1\right)\right) 
    \end{equation}
    locally uniformly on $U$. The sets $U$ and $U_N$ were defined in Proposition \ref{asymptoticprop}, $T_N$ was defined in Lemma \ref{sonD} and $a$ was defined in \eqref{a}. \label{asymptoticpartitionsfunction}
\end{Lem}
\begin{proof}
 By Lemma \ref{shiftintegration} we have 
    \begin{eqnarray}
        Z_{N,p,\beta}(h)&=&a^{N^{2}}\sqrt{\frac{N}{\beta_{\text{eff}} 2 \pi}}\int_{-\infty}^\infty \exp(N\Phi_N(s,h))\operatorname{d}s\nonumber\\
        &=&a^{N^{2}}\sqrt{\frac{N}{\beta_{\text{eff}} 2 \pi}}\int_{\gamma_N^\prime}\exp(N\Phi_N(s,h))\operatorname{d}s\nonumber\\
        &=&a^{N^{2}}\sqrt{\frac{N}{\beta_{\text{eff}} 2\pi}}\int_{-\infty}^\infty \exp(N\Phi_N(s_N(h)+\tau,h))\operatorname{d}\tau.\label{shifted}
    \end{eqnarray}
We want to prove for compact $K\subset U$ that
\begin{equation*}
    Q_{N,p,\beta}(h):=\frac{Z_{N,p,\beta}(h)}{a^{N^{2}}\sqrt{-\frac{1}{\beta_{\text{eff}}(\Phi_N)_{ss}(s_N(h),h)}}e^{N\Phi_N(s_N(h),h)}}=1 + o\left(1\right)
\end{equation*}
uniformly on $K$. As before, there exists $N_0\in\mathbb{N}$ such that $K\subset U_N$ for all $N\geq N_0$. By \eqref{shifted} we have
\begin{equation}
    Q_{N,p,\beta}(h)=\frac{\int_{-\infty}^\infty \exp(N\Phi_N(s_N(h)+\tau,h))\operatorname{d}\tau}{I_{\text{gauss}}}=:\frac{I}{I_{\text{gauss}}}, \label{fractionq}
\end{equation}
where
\begin{equation*}
    I_{\text{gauss}}:=\sqrt{\frac{2\pi}{-N(\Phi_N)_{ss}(s_N(h),h)}}e^{N\Phi_N(s_N(h),h)}.
\end{equation*}
We now split the domain of integration in the numerator of the right-hand side in \eqref{fractionq} into an inner part and a tail part. We write for $\rho>0$
\begin{eqnarray*}
    I&=&\int_{\vert\tau\vert\leq\rho}\exp(N\Phi_N(s_N(h)+\tau,h))\operatorname{d}\tau+\int_{\vert\tau\vert\geq\rho}\exp(N\Phi_N(s_N(h)+\tau,h))\operatorname{d}\tau\\
    &=:&I_{\text{inner}}+I_{\text{tail}}
\end{eqnarray*}
We choose $\rho$ later in the proof. We first treat $I_\text{inner}$:

For $\vert\tau\vert\leq \rho$ we have by Taylor expansion since $(\Phi_N)_s(s_N(h),h)=0$
\begin{equation}
    \Phi_N(s_N(h)+\tau,h)=\Phi_N(s_N(h),h)+\frac{1}{2}(\Phi_N)_{ss}(s_N(h),h)\tau^{2}+E_N(\tau,h),\label{taylorPhiN}
\end{equation}
where $E_N(\tau,h)$ is the remainder term for which there exists a constant $M>0$ such that uniformly on $K$ and all $N\in\mathbb{N}$ we have
\begin{equation}
    \vert E_N(\tau,h)\vert\leq M\vert \tau\vert^{3}. \label{boundonE}
\end{equation}
Putting \eqref{taylorPhiN} into $I_{\text{inner}}$ yields

\begin{align}
    &I_{\text{inner}}=\int_{-\rho}^{\rho} \exp\left(N\left( \Phi_N(s_N(h),h)+\frac{(\Phi_N)_{ss}(s_N(h),h)}{2}\tau^{2}+E_N(\tau,h)\right)\right)\operatorname{d}\tau\nonumber\\
    &=e^{N\Phi_N(s_N(h),h)}\left(\int_{-\rho}^{\rho} e^{\frac{N}{2}(\Phi_N)_{ss}(s_N(h),h)\tau^{2}}\operatorname{d}\tau\right.\nonumber\\
    &~~~~~~~~~~~~~~~~~~~~~~~~~~~~~~~~~~~~~~~~~~~~\left.+\int_{-\rho}^{\rho} e^{\frac{N}{2}(\Phi_N)_{ss}(s_N(h),h)\tau^{2}}(e^{N E_N(\tau,h)}-1)\operatorname{d}\tau\right). \label{Iinner}
\end{align}

We have
\begin{equation}
\begin{aligned}
    \int_{-\rho}^{\rho} e^{\frac{N}{2}(\Phi_N)_{ss}(s_N(h),h)\tau^{2}}\operatorname{d}\tau=&\int_{-\infty}^{\infty} e^{\frac{N}{2}(\Phi_N)_{ss}(s_N(h),h)\tau^{2}}\operatorname{d}\tau\\
    &-2\int_{\tau>\rho}e^{\frac{N}{2}(\Phi_N)_{ss}(s_N(h),h)\tau^{2}}\operatorname{d}\tau, \label{Gausspart}
\end{aligned}
\end{equation}
where the first integral in \eqref{Gausspart} is Gaussian and evaluates to $\sqrt{\frac{2\pi}{-N(\Phi_N)_{ss}(s_N(h),h)}}$. For the remaining integral, set $-b:=\sup_{N\geq N_0,h\in K}\operatorname{Re}(\Phi_N)_{ss}(s_N(h),h)$. 
By Lemma \ref{welldefsecderiv} we have that $-b<0$ if $N_0\in\mathbb{N}$ is large enough. 
We compute 
\begin{equation}
\begin{aligned}
\left|\int_{\tau\ge\rho} e^{\frac{N}{2}(\Phi_N)_{ss}(s_N(h),h)\tau^{2}}\operatorname{d}\tau\right|
&\le \int_{\tau\ge\rho} e^{\frac{N}{2}\operatorname{Re}(\Phi_N)_{ss}(s_N(h),h)\tau^{2}}\operatorname{d}\tau
\le \int_{\tau\ge\rho} e^{-\frac{N}{2}b\tau^{2}}\operatorname{d}\tau\\
&\le \int_{\tau\ge\rho} \frac{\tau}{\rho} e^{-\frac{N}{2}b\tau^{2}}\operatorname{d}\tau
= \frac{e^{-\frac{N}{2}b\rho^{2}}}{Nb\rho}. \label{tailbound}
\end{aligned}
\end{equation}
Hence we have by \eqref{Gausspart} 
\begin{equation}
    \int_{-\rho}^{\rho} e^{\frac{N}{2}(\Phi_N)_{ss}(s_N(h),h)\tau^{2}}\operatorname{d}\tau=\sqrt{\frac{2\pi}{-N(\Phi_N)_{ss}(s_N(h),h)}}+O\left(2\frac{e^{-\frac{N}{2}b\rho^{2}}}{Nb\rho}\right) \label{boundgauss}
\end{equation}
uniformly on $K$.
For the second integral in \eqref{Iinner}, set $\rho:=b/4M$ where $M$ is given by \eqref{boundonE}. For $\vert\tau\vert\leq \rho$ we have 
\begin{equation*}
\vert e^{NE_N(\tau,h)}-1\vert\leq\vert NE_N(\tau,h)\vert\int_0^{1}e^{tN\vert E_N(\tau,h)\vert}\operatorname{d}t\leq NM\vert\tau\vert^{3}e^{NM\vert \tau\vert^{3}}\leq NM\vert\tau\vert^{3}e^{\frac{Nb \tau^{2}}{4}}.
\end{equation*}
Thus we have
\begin{align}
&\left|\int_{-\rho}^{\rho}
e^{\frac{N}{2}(\Phi_N)_{ss}(s_N(h),h)\tau^{2}}
\bigl(e^{N E_N(\tau,h)}-1\bigr) \operatorname{d}\tau \right|
\leq
2NM \int_{0}^{\rho}
e^{-\frac{N}{2}b\tau^{2}}
\tau^{3}
e^{\frac{N}{4}b\tau^{2}}
 \operatorname{d}\tau\nonumber \\
&\leq
2NM \int_{0}^{\infty}
e^{-\frac{N}{4}b\tau^{2}}
\tau^{3}
\operatorname{d}\tau
=
\frac{16M}{Nb^{2}}=O\left(\frac{1}{N}\right)
\label{innerbound}
\end{align}
uniformly in $h$ on $K$.
Putting \eqref{boundgauss} and \eqref{innerbound}  into \eqref{Iinner} we find
\begin{equation*}
    I_{\text{inner}}=I_{\text{gauss}}+O\left(\frac{e^{N\Phi_N(s_N(h),h)}e^{-\frac{N}{2}b\rho^{2}}}{N}\right)+O\left(\frac{e^{N\Phi_N(s_N(h),h)}}{N}\right)
\end{equation*}
and thus
\begin{equation}
     \frac{I_{\text{inner}}}{I_{\text{gauss}}}=1+O\left(\frac{1}{\sqrt{N}}\right)=1+o(1) \label{asymptoticinner}
\end{equation}
uniformly on $K$.

Now for the tail: We only treat the integral for $\tau>\rho$. The integral for $\tau<-\rho$ is treated analogously. We have
\begin{equation*}
    \int_{\rho}^\infty e^{N\Phi_N(s_N(h)+\tau,h)}\operatorname{d}\tau=e^{N\Phi_N(s_N(h),h)}\int_{\rho}^\infty e^{N(\Phi_N(s_N(h)+\tau,h)-\Phi_N(s_N(h),h))}\operatorname{d}\tau.
\end{equation*}
We want to show that the remaining integral goes to zero exponentially fast. Observe, again by Taylor's theorem and Lemma \ref{welldefsecderiv},
\begin{equation*}
    \operatorname{Re}\biggl(\Phi_N(s_N(h)+\tau,h)-\Phi_N(s_N(h),h)\biggr)\leq -\frac{M^\prime}{2}\vert\tau\vert^{2}
\end{equation*}
for some $M^\prime>0$ uniformly on $K$ and for all $N\geq N_0$. Now we have
\begin{eqnarray*}
    \left\vert\int_{\rho}^\infty e^{N(\Phi_N(s_N(h)+\tau,h)-\Phi_N(s_N(h),h))}\operatorname{d}\tau\right\vert&\leq& \int_{\rho}^\infty e^{-\frac{N}{2}M^\prime\tau^{2}}\operatorname{d}\tau\\
    &\leq& \frac{e^{-\frac{N}{2}M^\prime\rho^{2}}}{NM^\prime\rho}=O\left(\frac{e^{-\frac{N}{2}M^\prime\rho^{2}}}{N}\right)
\end{eqnarray*}
as in \eqref{tailbound}, where the $O$-term is uniform in $h$ on $K$. Now this implies
\begin{equation}
    \left\vert\frac{I_{\text{tail}}}{I_{\text{gauss}}}\right\vert=O\left(\frac{e^{-\frac{N}{2}M^\prime\rho^{2}}}{\sqrt{N}}\right)=o(1)\label{asymptotictail}
\end{equation}
uniformly on $K$. Plugging \eqref{asymptoticinner} and \eqref{asymptotictail} into \eqref{fractionq} yields
\begin{equation*}
    Q_{N,p,\beta}(h)=1+O\left(\sqrt{\frac{1}{N}}\right)=1+o(1),
\end{equation*}
where the implicit constant in the $o$-term is locally uniform in $h$.
\end{proof}
\begin{proof}[Proof of Proposition~\ref{asymptoticprop}]
By Lemma \ref{sonD} there exists a holomorphic function $s_N\colon U_N\to T_N$ and a holomorphic, non-vanishing function
\begin{equation*}
A_N(h)
:= \sqrt{-\frac{1}{\beta_{\text{eff}}(\Phi_N)_{ss}(s_N(h),h)}},
\end{equation*}
which is well defined by Lemma \ref{welldefsecderiv} and satisfies by Lemma \ref{asymptoticpartitionsfunction}
\begin{equation*}
Z_{N,p,\beta}(h)
= a^{N^{2}}A_N(h)e^{N\Phi_N(s_N(h),h)}\left(1+o(1)\right)
\end{equation*}
locally uniformly on $U$.
Since $(\Phi_N)_{ss}(s_N(h),h)$ stays bounded away from $0$ on any compact
set $K\subset U$ by Lemma \ref{welldefsecderiv}, the function $A_N(h)$ is holomorphic and bounded on $K$. In particular, we may choose a holomorphic branch of
$\log A_N(h)$ on $K$ by Theorem 6.2 in Chapter 3 of \cite{stein2010complex}. Taking the logarithm in the representation above yields
\begin{equation*}
\log Z_{N,p,\beta}(h)
  =\log\left(a^{N^{2}}\right)+ N\Phi_N(s_N(h),h) + \log A_N(h) + \log(1+R_N(h)),
\end{equation*}
where $R_N(h)=o(1)$ locally uniformly.
Using the definition of the finite–volume pressure in \eqref{finitevolumepress}, we obtain
\begin{equation*}
\psi_{N,p,\beta}(h)
= \frac{\log\left(a^{N^{2}}\right)}{N}+\Phi_N(s_N(h),h)
  + \frac{1}{N}\log A_N(h)
  + \frac{1}{N}\log(1+R_N(h)).
\end{equation*}

On each compact $K\subset U$, the function $\log A_N(h)$ is bounded uniformly for all $N$, hence uniformly on $K$,
\begin{equation*}
\frac{1}{N}\log A_N(h)
  = O\left(\frac{1}{N}\right).
\end{equation*}
Moreover, since $R_N(h)=o(1)$ locally uniformly, we have
for $N$ sufficiently large that $|R_N(h)|\le 1/2$ on $K$, and therefore by Taylor expansion, the triangle inequality and the geometric series that 
\begin{equation*}
|\log(1+R_N(h))|\leq\sum_{n=1}^\infty\vert R_N(h)\vert^{n}=\frac{\vert R_N(h)\vert}{1-\vert R_N(h)\vert}
  \le 2|R_N(h)|
  = o(1),
\end{equation*}
implying
\begin{equation*}
\frac{1}{N}\log(1+R_N(h))
  = O\left(\frac{1}{N}\right)
\end{equation*}
locally uniformly in $h$.

Now since $p^{3}N^{2}\to\infty$ also implies $pN\to\infty$ as $N\to\infty$ we have by Remark \ref{convergenceaN2}
\begin{equation*}
    \frac{\log\left(a^{N^{2}}\right)}{N}=O\left(\frac{1}{pN}\right) ~~~\text{ as }N\to\infty.
\end{equation*}
Combining these estimates, we conclude since $\Phi_N(s_N(h),h)\to\Phi(s(h),h)$ locally uniformly on $U$ by Lemma \ref{uniforms} that
\begin{equation*}
\psi_{N,p,\beta}(h)
  = \Phi(s(h),h) + O\!\left(\frac{1}{pN}\right),
\end{equation*}
locally uniformly on $U$. This completes the proof of Proposition \ref{asymptoticprop}.
\end{proof}

Hence, we have completed the proof of the locally uniform convergence of $\psi_{N,p,\beta}(h)$ to $\Phi(s(h),h)$ (Proposition \ref{asymptoticprop}). The proof of the sharp cumulant bounds stated in Theorem \ref{Statuleviciusthmcw} assuming Proposition \ref{asymptoticprop} was already provided in Section \ref{Proofofthm}. By the refined method of cumulants, this directly implies a number of quantitative results stated in Corollary \ref{Implications}.

\section*{Acknowledgment} The authors have been supported by the DFG (German Research Foundation) Priority Program
”Random Geometric Systems” (SPP 2265) - project number: 531542011. The authors are grateful to the anonymous referees for their careful reading and their valuable comments and suggestions leading to an substantial improvement of the discussion. ChatGPT5 was used as an auxiliary editorial tool and literature search assistance. All references and final wording were checked and approved by the authors.
\section*{Funding Acknowledgement}
The authors have been funded by the DFG (German Research Foundation) Priority Programm "Random
Geometric Systems" (SPP 2265) - project number: 531542011.
\printbibliography
\end{document}